\documentclass[11pt]{article}
\usepackage[a4paper,top=2cm,bottom=2cm,left=1.5cm,right=1.5cm]{geometry}
\usepackage[utf8]{inputenc}
\usepackage{natbib}
\usepackage{array}
\usepackage{tikz}
\usepackage{stmaryrd}
\usepackage{placeins}
\usepackage{amsthm}
\usepackage{amsmath,amsfonts,amssymb}
\usepackage{graphicx}
\graphicspath{{./img/}}
\usepackage{float}

\newtheorem{theo}{Theorem}[section] 
\newtheorem{lem}{Lemma}[section]

\DeclareUnicodeCharacter{B0}{\textdegree}
\usepackage[draft]{fixme}
\fxusetheme{color}
\FXRegisterAuthor{h}{ah}{H}
\FXRegisterAuthor{p}{ap}{P}
\FXRegisterAuthor{a}{aa}{A}
\usepackage{hyperref}
\hypersetup{
pdfpagemode=UseOutlines,      
pdfstartview=Fit,             
pdffitwindow=true,            
pdfpagelayout=TwoColumnsRight,
pdftoolbar=true,              
pdfmenubar=true,              
bookmarksopen=false,          
bookmarksnumbered=true,       
colorlinks=true,              
pdfauthor={Ton nom},     
pdftitle={Titre PDF},    
pdfcreator=PDFLaTeX,          %
pdfproducer=PDFLaTeX,         %
linkcolor=blue,               
urlcolor=blue,                
anchorcolor=blue,            
citecolor=blue,              
frenchlinks=true,             
pdfborder={0 0 0}             
}
\usepackage[normalem]{ulem}
\usepackage{enumitem}
\usepackage{amsmath,amssymb,color}
\usepackage[utf8]{inputenc} 
\usepackage{hyperref}
\input epsf
\def\build#1_#2^#3{\mathrel{
\mathop{\kern 0pt#1}\limits_{#2}^{#3}}}

\def\Annexe{{\large Annexe}}
\newcommand{\bqa}{\begin{eqnarray}}
\newcommand{\eqa}{\end{eqnarray}}
\newcommand{\bdesc}{\begin{description}}
\newcommand{\edesc}{\end{description}}
\newcommand{\bqan}{\begin{eqnarray*}}
\newcommand{\eqan}{\end{eqnarray*}}

\def\norm#1{\left\| #1 \right\|}
\def\abs#1{\left| #1 \right|}

\def\pa#1{\left( #1 \right)}

\def\cro#1{\left[ #1 \right]}
\def\acco#1{\left \{  #1\right \}}

\def\ndx2{n\,d_x^2}

\def\bkF{\mbox{{\rm I\kern-.17em F}}}

\def\bkNp{\mbox{{\small\rm I\kern-.20em N}}}
\def\bkO{{\rm \kern.24em
            \vrule width.05em height1.4ex depth-.05ex
            \kern-.26em O}}
\def\bkZ{\mbox{{\rm Z\kern-.32em Z}}}

\def\non{\mbox{$^{\mbox{\bf --}}$\kern-.1em \vrule width.05em height1.5ex}}

\def\emptysq{\mathbin{\vbox{\hrule\hbox{\vrule height1ex \kern.5em 
       \vrule height1ex}\hrule}}}

\usepackage{indentfirst}
\usepackage{bm}
\usepackage{multirow}
\usepackage{float}
\usepackage{bbm}

\usepackage{braket}

\newtheorem{assumption}{Assumption}

\begin{document}

\footnotetext[1]{antoine.godichon$\_$baggioni@upmc.fr,   
	  Laboratoire de Probabilités, Statistique et Modélisation (LPSM),
	Sorbonne Université,
	4 Place Jussieu, 75005 Paris, France.}
\footnotetext[2]{Laboratoire de Mathématiques de l'INSA Rouen Normandie,
	INSA Rouen Normandie,
	BP 08 - Avenue de l'Université, 76800 Saint-Etienne du Rouvray, France}

\date{}
\title{A Full Adagrad algorithm with $O(Nd)$ operations}

\author{ Antoine Godichon-Baggioni   \footnotemark[1],	 
  Wei Lu\footnotemark[2] and 
	  Bruno Portier\footnotemark[2]
}

\maketitle

\begin{abstract}
A novel approach is given to overcome the computational challenges of the full-matrix Adaptive Gradient algorithm (Full AdaGrad) in stochastic optimization. By developing a recursive method that estimates the inverse of the square root of the covariance of the gradient, alongside a streaming variant for parameter updates, the study offers efficient and practical algorithms for large-scale applications. This innovative strategy significantly reduces the complexity and resource demands typically associated with full-matrix methods, enabling more effective optimization processes. Moreover, the convergence rates of the proposed estimators and their asymptotic efficiency are given. Their effectiveness is demonstrated through numerical studies.
\end{abstract}


\section{Introduction}\label{sec::intro}

Stochastic optimization plays a crucial role in machine learning and data science, particularly relevant in the context of high-dimensional data \citep{genevay2016stochastic,bottou2018optimization,sun2019survey}. 
This paper focuses on the stochastic gradient-based methods.
It targets on a scalar objective function $f(X,\theta)$, where $X$ is a random variable taking values in a measurable space $\mathcal{X}$  and $\theta$ is a parameter vector in $\mathbb{R}^d$.
This function is assumed to be differentiable with respect to $\theta$. 
Our goal is to minimize the expected value of this function, 
denoted as $F(\theta):=\mathbb{E}\cro{f(X,\theta)}$, in relation to $\theta$. 
The realizations of $X$ at different time steps are denoted as $X_1, \cdots,X_t, \cdots$, 
and $g_t(\theta):=\nabla_\theta f(X_t,\theta)$ refers to the gradient of $f(X_t,\cdot)$.

A popular approach in addressing this problem of optimization is Stochastic Gradient Descent(SGD), introduced by \cite{robbins1951stochastic}. 
It  recursively updates the parameter estimate based on the last estimate of the gradient, i.e.
$$\theta_{t} = \theta_{t-1} - \nu_tg_t(\theta_{t-1}),$$
where $\nu_t$ is the learning rate and $\theta_0$ is arbitrarily chosen. 
Despite its computational efficiency and favorable convergence properties, 
SGD faces limitations, particularly in adapting the learning rate to the varying scales of features \citep{ruder2016overview}.

To address these limitations, many extensions of SGD have been proposed
. A widely used variant is the Adaptive Gradient algorithm (AdaGrad) introduced by \cite{duchi2011adaptive}. 
It adapts the learning rate for each parameter, offering improved performances on problems with sparse gradients. 
The full-matrix version of AdaGrad can be expressed as follows:
$$\theta_{t}= \theta_{t-1}-\nu_t\mathcal{G}_t^{-1/2}g_t(\theta_{t-1}),$$
where $\mathcal{G}_t:= \sum^t_{k=1}g_k(\theta_{k-1})g_k(\theta_{k-1})^T$ 
is a recursive estimate of the covariance matrix of the gradient 
and $\mathcal{G}_t^{-1/2}$ is the inverse of the square root of it. 
{
However, one major challenge of AdaGrad is that for every new data point, it is necessary to compute the square root of the inverse of $G_t$ in an online manner, which is computationally heavy.
}
This computation is particularly demanding in terms of computational resources, with a complexity of order $\mathcal{O}(d^3)$. 
Such complexity is often prohibitive, especially in scenarios involving -dimensional data. 
To deal with it, a diagonal version of AdaGrad was proposed, simplifying the process by using only the diagonal elements of $\mathcal{G}_t$, i.e
\begin{equation}\label{DefAdaDiag}
	\theta_{t}= \theta_{t-1}-\nu_t \text{diag}\pa{\mathcal{G}_t}^{-1/2}g_t(\theta_{t-1}).
\end{equation}
In practice, this approach is more feasible and  is broadly applied for machine learning tasks \citep{dean2012large,seide20141,smith2017cyclical}. 
Furthermore, \cite{defossez2022a}  establishes the standard convergence rate for Adagrad in the non convex case.
Despite being more practical, the diagonal version of AdaGrad inherently loses information compared to the full-matrix version, especially in the case where the gradient have coordinates highly correlated.  
\vspace{1ex}

Our work focuses on the full-matrix version of AdaGrad, 
proposing a recursive method to estimate the inverse of the square root of the covariance matrix $\Sigma := \mathbb{E}\cro{\nabla_\theta f(X,\theta^*)\nabla_\theta f(X,\theta^*)^T},$ 
where $\theta^*$ minimizes the function $F$. 
Unlike the original Full AdaGrad, which uses $G_t$ to estimate $\Sigma$ and then computes $G_t^{-1/2}$, 
we will directly estimate $\Sigma^{-1/2}$.
Using the fact that 
\[
\Sigma^{-1/2}\Sigma\Sigma^{-1/2} - I_d = \mathbb{E}\cro{\Sigma^{-1/2}\nabla_\theta f(X,\theta^*)\nabla_\theta f(X,\theta^*)^T \Sigma^{-1/2} - I_{d}} = 0,
\]
we introduce a Robbins-Monro algorithm to 
{recursively (or online)}
 estimate $\Sigma^{-1/2}$. 
This estimator,  denoted as $A_t$, is defined for all $t\geq 1$, by:
\begin{align*}
	A_{t} = A_{t-1}-\gamma_{t}\pa{A_{t-1}g_t(\theta_{t-1})g_t(\theta_{t-1})^TA_{t-1}-I_d}, \label{def::AtSansIndication}
\end{align*}
where $A_0 = I_d$ 
and $(\gamma_t)_{t \geq 1}$ is a sequence of positive real numbers, decreasing towards 0. 
This estimate is used in updating the estimate of $\theta$:
\begin{align*}
	\theta_{t} = \theta_{t-1}-\nu_{t}A_{t-1}g_t(\theta_{t-1}).
\end{align*}
Consequently, this approach enables us to avoid the expensive computation of the square root of the inverse of $\mathcal{G}_t$, enhancing the computational efficiency of the algorithm. 
Nevertheless, $A_{t}$ is not necessarily positive definite, 
and we so propose a slight modification in this sense. 
In addition, $\theta_{t}$ cannot be asymptotically efficient, 
and we so introduced its (weighted) averaged version \citep{polyak1992acceleration,pelletier2000asymptotic,mokkadem2011generalization,boyer2023asymptotic}.

Although the propose approach to estimate $\Sigma^{-1/2}$ enables to reduce the calculus time, this only enables to achieve a total complexity of order $O(Nd^{2})$, where $N$ is the sample size. 
Then, we propose a Streaming {(also called online mini-batch)} version of our algorithm, updating the estimate of $\Sigma^{-1/2}$ and $\theta$ only after observing every $n$ gradients and using their average. 
This approach further reduces the algorithm's complexity, making it more practical for large-scale applications. 
More precisely, a good choice of $n$ ($n=d$ for instance) enables to obtain asymptotically efficient estimates with a complexity of order $O(Nd)$, i.e. with the same complexity as for Adagrad algorithm.
\vspace{1ex}

The paper is organized as follows. 
The general framework is introduced in   Section \ref{sec::framewk}. 
In Section \ref{sec::FullNonStr}, we present a detailed description of the proposed Averaged Full AdaGrad algorithm before establishing  its asymptotic efficiency. 
Following this, we introduce a streaming variant of the Full AdaGrad algorithm in \ref{sec::FullStreaming} and we obtain the asymptotic efficiency of the proposed estimates. In Section \ref{sec::simu}, we illustrate the practical applicability of our algorithms through numerical studies. The proofs are postponed in Section \ref{proof}.


\section{Framework}\label{sec::framewk}
Let us recall that the aim is to minimize the functional $F : \mathbb{R}^d \longrightarrow \mathbb{R}$ defined for all $\theta \in \mathbb{R}^d$ by:
$$F(\theta):=\mathbb{E}\cro{f(X,\theta)}, $$
where $f: \mathcal{X} \times \mathbb{R}^d \to \mathbb{R}$.
In all the sequel, we suppose that the following assumptions are fulfilled:
\begin{assumption}\label{ass1}
{The function $F$ is convex, twice continuously differentiable, there is $\theta^* \in \mathbb{R}^d$ such that $\nabla F(\theta^*)=0$, and is locally strongly convex, i.e $\nabla^{2}G(\theta^{*})$ is positive.} 
\end{assumption}
This assumption ensures that $\theta^{*}$ is the unique minimizer of the functional $F$ and legitimates the use of gradient-type methods.
 
\begin{assumption}\label{ass::moment}
	There exists an integer $p \ge 1$ and a positive constant $C_p < \infty$ such that for all $\theta \in \mathbb{R}^d$
	\[
	\mathbb{E}\left[ \left\|\nabla_{\theta}f(X,\theta)\right\|^{2p} \right] \leq C_{p} + C_p \left( F(\theta)- F(\theta^{*}) \right)^{p}.
	\]
\end{assumption}
In the literature on stochastic gradient algorithms, 
it is common to consider moments of order 2 ($p=1$) or 4 ($p=2$) for the gradient of $f$ (see, e.g., \cite{pelletier1998almost,pelletier2000asymptotic}).
However,  due to some hyperparameters within our algorithm,
we must strongly constraint the moment order of the gradient of $f$ when determining the convergence rate of our estimates.
The specific value of $p$ will be delineated in the theorem statements. {
Observe that one could consider the following alternative assumption and all the stated results remain true:
there are positive constant $C_{1},C_{p},C_{M}$ such that for all $\theta \in \mathbb{R}^{d}$,
\[
\mathbb{E} \left[ \left\|  \nabla_{\theta}f\left( X, \theta \right) \right\|^{2} \right] \leq C_{1} + C_{1} \left( F(\theta) - F(\theta^{*} ) \right) \quad \quad \text{and} \quad \quad \mathbb{E} \left[ \left\| \nabla_{\theta}f \left( X, \theta \right) \right\|^{2p}\mathbf{1}_{\left\| \theta - \theta^{*} \right\| \leq C_M } \right] \leq C_{p} .
\]
}
\begin{assumption}\label{ass::sigma::lip}
	The function $\Sigma: \theta \longmapsto \mathbb{E}\left[  \nabla_{\theta}f(X,\theta) \nabla_{\theta}f(X,\theta)^{T} \right]$ is $L_{\Sigma}$-Lipschitz and $\Sigma \left( \theta^{*} \right)$ is positive.
\end{assumption}
This assumption is quite specific to our work on FullAdagrad as it ensures the convergence of estimates of the  variance, and more specifically in our case, of the square root of their inverse. It is worth noting that this assumption is quite common in the literature, particularly when considering the estimation of asymptotic covariance \citep{zhu2023online,GBL2024}.

The above are assumptions regarding the first-order derivatives of $F$. 
Next, we present some necessary assumptions concerning the second-order derivatives of the function.

\begin{assumption}\label{ass::hess}
The Hessian of $F$ is uniformly bounded by $L_{\nabla F}$.
\end{assumption}

This assumption ensures that the gradient of $F$ is $L_{\nabla F}$-Lipschitz which is crucial to obtain the consistency of the estimates (via a Taylor's expansion of the gradient at order $2$).

\begin{assumption}\label{ass::hess::lip}
The Hessian of $F$ is Locally Lipschitz: there exists $\eta>0  $ and $L_{\eta}>0$ such that for all $\theta \in \mathcal{B}\left( \theta^{*} , \eta \right)$,
\[
\left\| \nabla F(\theta) - \nabla^{2}F \left( \theta^{*} \right) \left( \theta - \theta^{*} \right) \right\| \leq L_{\eta} \left\| \theta - \theta^{*} \right\|^{2}.
\]
\end{assumption}
These assumptions are close to those found in the literature \citep{pelletier2000asymptotic,GADAT2023312,boyer2023asymptotic}. 
The main differences come from Assumption \ref{ass::moment} and \ref{ass::sigma::lip}. These last ones are crucial for the theoretical study of the estimates of $\Sigma^{-1/2}$, i.e. to prove their strong consistency.

\section{A Full AdaGrad algorithm with  $\mathcal{O}(td^{2})$ operations}\label{sec::FullNonStr}
In this section, we introduce a Full AdaGrad algorithm with  $\mathcal{O}(td^{2})$ operations. We focus on recursively estimating $\Sigma^{-1/2}$ using a Robbins-Monro algorithm, in order to refine estimates of $\theta^*$ while ensuring computational performance. 

\subsection{Estimating $\Sigma^{-1/2}$ with the help of a Robbins-Monro algorithm}
First, we focus on  recursive estimates of the matrix $\Sigma^{-1/2}$.  In all the sequel, let  
 $X_1,\ldots , X_t,\ldots$ be  i.i.d. copies of $X$ 
and for all $\theta \in \mathbb{R}^d$, we denote $g_t(\theta):=\nabla_\theta f(X_t,\theta)$. 
Let us recall that the Robbins-Monro algorithm  for estimating $\Sigma^{-1/2}$, described in the Introduction, is defined recursively for all $t \geq 0$ by
$$	A_{t+1} = A_{t}-\gamma_{t+1}\pa{A_{t}g_{t+1}(\theta_{t})g_{t+1}(\theta_{t})^TA_{t}-I_d}, $$
where $A_0$ is a symmetric positive definite matrix, $(\theta_{t})_{t\ge0}$ is a sequence of estimates of $\theta^*$, and $\gamma_{t}=c_\gamma t^{-\gamma}$ with $c_\gamma >0$ and $1/2<\gamma<1$.
Observe that $A_{t+1}g_t(\theta_{t})$ is a vector, implying that the complexity of this operation is of order $O(d^2)$. However, we cannot ensure that the matrix $A_t$ is always positive definite. 
Nevertheless, in Full AdaGrad, $A_t$ must always be positive to guarantee that at each step, we go in the direction of the gradient (in average). 
To address this issue, we propose a slightly modified version of $A_t$  defined for all $t \geq 0$ by
\begin{equation*}
	A_{t+1} = A_t - \gamma_{t+1}\pa{A_tg_{t+1}\pa{\theta_{t}}g_{t+1}\pa{\theta_{t}}^TA_t -Id}\mathbf{1}_{\acco{g_{t+1}\pa{\theta_{t}}^TA_tg_{t+1}\pa{\theta_{t}}\le\beta_{t+1}}},
\end{equation*}
where $\beta_t = c_\beta t^\beta$ with $0<\beta<1/2$ and $0<c_\beta c_\gamma<1$. 
In fact, $g_{t+1}(\theta_{t})^T A_t g_{t+1}(\theta_{t})$ is the unique positive eigenvalue of the rank-1 matrix $A_t g_{t+1}(\theta_{t}) g_{t+1}(\theta_{t})^T$. 
We update $A_t$ only when this value is not excessively large and thanks to this modification, $A_t$ is positive definite for any $t \geq 0$.

 \subsection{Full AdaGrad algorithms with  $\mathcal{O}(td^{2})$ operations}
We can now propose a Full AdaGrad algorithm defined for all $t \geq 0$ by
\begin{align}
	\theta_{t+1} &= \theta_{t} - \nu_{t+1}A_tg_{t+1}\pa{\theta_{t}} \label{thetaTNoWA}, \\
	A_{t+1} &= A_t - \gamma_{t+1}\pa{  A_tg_{t+1}\pa{\theta_{t}}g_{t+1}\pa{\theta_{t}}^TA_t -Id}\mathbf{1}_{\acco{g_{t+1}\pa{\theta_{t}}^TA_tg_{t+1}\pa{\theta_{t}}\le\beta_{t+1}}}, \label{ATNoWA}
\end{align}
where $\theta_0$ is arbitrarily chosen.
Although our numerical studies show that this algorithm performs well (see Section \ref{sec::simu}), the obtained estimates are not asymptotically efficient. 
Therefore, to ensure  the asymptotic optimality  of the estimates, and to enhance the performance of the algorithm  in practice, we  follow  the idea of \cite{mokkadem2011generalization,boyer2023asymptotic}. 
More precisely,  we introduce the Weighted Averaged Full AdaGrad (WAFA for short) defined recursively  for all $t \geq 0$ by

\begin{align}
	\theta_{t+1} &= \theta_{t} - \nu_{t+1}{\hat{A_t}}g_{t+1}\pa{\theta_{t}} \label{thetaT} \\ 
	\theta_{t+1,\tau} &= \pa{1-\frac{\ln (t+1)^{\tau}}{\sum_{k=0}^t \ln (k+1)^{\tau}}}\theta_{t,\tau}+\frac{\ln (t+1)^{\tau}}{\sum_{k=0}^t \ln (k+1)^{\tau}}\theta_{t+1} \label{thetaTAU}\\ 
	A_{t+1} &= A_t - \gamma_{t+1}\pa{  A_tg_{t+1}\pa{\theta_{t,\tau}}g_{t+1}\pa{\theta_{t,\tau}}^TA_t -Id}\mathbf{1}_{\acco{g_{t+1}\pa{\theta_{t,\tau}}^TA_tg_{t+1}\pa{\theta_{t,\tau}}\le\beta_{t+1}}} \label{AT}\\
	A_{t+1,\tau '} &= \pa{1-\frac{\ln (t+1)^{\tau'}}{\sum_{k=0}^t \ln (k+1)^{\tau'}}}A_{t,\tau'}+\frac{\ln (t+1)^{\tau'}}{\sum_{k=0}^t \ln (k+1)^{\tau'}}A_{t+1} \label{ATAU}
\end{align}
with $\theta_{0,\tau}=\theta_0$, $A_{0,\tau'}=A_0$ and $\tau,\tau' \geq 0$. {{Note that one can take \(\hat{A}_t = A_t\) or \(A_{t,\tau}\). The proofs are provided only for the case \(\hat{A}_t = A_{t,\tau}\), but the other case follows straightforwardly. Using the weighted average estimates \(A_{t,\tau}\) should help accelerate convergence, though we do not establish this result formally.
}} Finally, observe that when $\tau, \tau'=0$, we obtain the usual averaged estimates. However, taking both  greater than zero allows to place more weight on the recent estimations, which are supposed to be better.   The following theorem gives the strong consistency of the Full Adagrad estimates of $\theta^*$.

\begin{theo}\label{theo::consistency}
Suppose Assumptions \ref{ass1}, \ref{ass::moment}   and \ref{ass::hess} hold. Suppose also that $ 2\gamma +2\nu > 3$ and $\nu + \beta < 1$.
 Then  $\theta_{t}$ and $\theta_{t,\tau}$ defined by \eqref{thetaT} and \eqref{thetaTAU} converge almost surely to $\theta^{*}$.
\end{theo}
The proof is given in Section \ref{proof}. 
The hyperparameters constraints introduced here are for technical reasons. These conditions are not necessary in practice (see Section \ref{sec::simu}).
In the following theorem, we establish the strong consistency of the estimates of $\Sigma^{-1}$ and the almost sure convergence rates of the estimates of $\theta^*$. 
\begin{theo}\label{theo::rate}
Suppose Assumptions \ref{ass1}, \ref{ass::sigma::lip}    and \ref{ass::hess}  hold as well as \ref{ass::moment} with $p> \max \left\lbrace  \frac{8-8\gamma}{\gamma + \beta -1} , 2\left( \frac{1}{\gamma}-1 \right) \right\rbrace$. Suppose also that $ 2\gamma +2\nu > 3$, $\nu + \beta < 1$, $2\gamma -2\beta > 1$ and that $\gamma + \beta > 1$.  Then
\[
A_{t} \xrightarrow[t\to + \infty]{a.s} \Sigma^{-1/2}  \quad \quad \text{ and } \quad \quad A_{t,\tau '} \xrightarrow[t\to + \infty]{a.s} \Sigma^{-1/2} 
\]
In addition, 
\[
\left\| \theta_{t} - \theta^{*} \right\|^{2} = O \left( \frac{\ln t}{t^{\nu}} \right) \quad a.s.  \quad \quad \text{ and } \quad \quad \left\| \theta_{t,\tau} - \theta^{*} \right\|^{2} = O \left( \frac{\ln t}{t^{\nu}} \right) \quad a.s.
\]
\end{theo}
The proof is given in Section \ref{proof}. 
Observe that the conditions on $\gamma , \nu , \beta$ imply that $\nu < \gamma$ and $\gamma > 3/4$. 
These conditions are due to the use of Robbins-Siegmund Theorem 
and should be certainly improved. 
Indeed, we will see in Section \ref{sec::simu} that these conditions do not need to be fulfilled in practice. Finally, under slightly restricted conditions, the following theorem gives  better convergence rates of $\theta^*$.
\begin{theo}\label{theo::tlc}
Suppose Assumptions \ref{ass1}, \ref{ass::sigma::lip}, \ref{ass::hess} and \ref{ass::hess::lip}    hold as well as \ref{ass::moment} with 
$p> \max \left\lbrace  \frac{8-8\gamma}{\gamma + \beta -1} , 2\left( \frac{1}{\gamma}-1 \right) \right\rbrace$. 
Suppose also that $ 2\gamma +2\nu > 3$, $\nu + \beta < 1$, $2\gamma -2\beta > 1$ and that $\gamma + \beta > 1$.  
Then,
\[
\left\| \theta_{t,\tau} - \theta^{*} \right\|^{2} = O \left( \frac{\ln t}{t} \right) \quad a.s. \quad \quad \text{ and } \quad \quad \sqrt{ t} \left( \theta_{t,\tau} - \theta^{*} \right) \xrightarrow[t \to + \infty]{\mathcal{L}} \mathcal{N} \left( 0 , H^{-1}\Sigma H^{-1} \right)
\]
with $\Sigma := \Sigma(\theta^*)$ and $H:= \nabla^2 F (\theta^*)$.
\end{theo}
The proof is given in Section \ref{proof}. 
Thus, we obtain the asymptotic efficiency of the weighted average estimator. Furthermore, these estimators require only \(O(Nd^{2})\) operations, compared to a complexity of \(O(Nd^{3})\) operations if we directly compute \(\mathcal{G}_{t}^{-1/2}\).

\section{A Streaming Full AdaGrad algorithm with  $\mathcal{O} (N_{t}d)$ operations}\label{sec::FullStreaming}
In this section, following the idea of \cite{GBW2023OND}, we introduce a Streaming Weighted Averaged Full AdaGrad algorithm (SWAFA for short) to reduce the computational complexity of the algorithm. 
We consider that samples arrive (or are dealt with) by blocks of size $n \in \mathbb{N}$. 
More precisely, we suppose that at time $t$, we have $n$ new  i.i.d copies of $X$ denoted as $(X_{t,1},\ldots ,X_{t,n})$. 
Therefore, at time $t$, we will have observed a total of $N_t = nt$ i.i.d copies of $X$.
\vspace{1ex}

In this scenario, let us denote
$g_{t+1}(\theta_{t}) = \frac{1}{n} \sum_{i=1}^{n} \nabla_{\theta} f(X_{t+1,i},\theta_{t})$. Then, the streaming algorithm is defined recursively for all $t\ge0$ by
\begin{align}
	\theta_{t+1} &= \theta_{t} - \nu_{t+1}{\hat{A_t}}g_{t+1}\pa{\theta_{t}} \label{thetaT::stream} \\ 
	\theta_{t+1,\tau} &= \pa{1-\frac{\ln (t+1)^{\tau}}{\sum_{k=0}^t \ln (k+1)^{\tau}}}\theta_{t,\tau}+\frac{\ln (t+1)^{\tau}}{\sum_{k=0}^t \ln (k+1)^{\tau}}\theta_{t+1} \label{thetaTAU::stream}\\ 
	A_{t+1} &= A_t - \gamma_{t+1}\pa{ n A_tg_{t+1}\pa{\theta_{t,\tau}}g_{t+1}\pa{\theta_{t,\tau}}^TA_t -Id}\mathbf{1}_{\acco{ n g_{t+1}\pa{\theta_{t,\tau}}^TA_tg_{t+1}\pa{\theta_{t,\tau}}\le\beta_{t+1}}} \label{AT::stream}\\
	A_{t+1,\tau '} &= \pa{1-\frac{\ln (t+1)^{\tau'}}{\sum_{k=0}^t \ln (k+1)^{\tau'}}}A_{t,\tau'}+\frac{\ln (t+1)^{\tau'}}{\sum_{k=0}^t \ln (k+1)^{\tau'}}A_{t+1} \label{ATAU::stream}
\end{align}
Then, we still have $O(d^{2})$ operations for updating $A_{t},A_{t,\tau '}$ and $\theta_{t}$. Nevertheless, we only have $t = \frac{N_{t}}{n}$ iterations. This leads to total number of operations of order $O (N_{t}d^{2}n^{-1})$ operations detailed as follows:
\[
\underbrace{ N_{t}d +  \frac{N_{t}d^{2}}{n}}_{\text{updating } \theta_{t},A_{t},A_{t,\tau' } } +\underbrace{ \frac{N_{t}d}{n} }_{\text{updating } \theta_{t,\tau} } .
\]

Considering $n = d$ enables the complexity of the algorithm to be reduced to $\mathcal{O}(N_{t}d)$ operations, which is equivalent to the complexity of the AdaGrad algorithm   defined by \eqref{DefAdaDiag} {(up to the computing of the inverse square root for Adagrad)}.  {Note that we use the term 'streaming' here, but one could also use 'online mini-batch' instead.} We next give three theorems that establish the strong consistency, convergence rates, and asymptotic efficiency of the SWAFA estimates.
\begin{theo}\label{theo::consistency::stream}
Suppose Assumptions \ref{ass1}  and \ref{ass::hess} hold. Suppose also that $ 2\gamma +2\nu > 3$ and $\nu + \beta < 1$. Then  $\theta_{t}$ and $\theta_{t,\tau}$ defined by \eqref{thetaT::stream} and \eqref{thetaTAU::stream} converge almost surely to $\theta^{*}$.
\end{theo}
The proof is very similar to the one of Theorem \ref{theo::consistency} and is therefore not given. 

\begin{theo}\label{theo::rate::stream}
Suppose Assumptions \ref{ass1}, \ref{ass::sigma::lip}  and \ref{ass::hess}  hold as well as \ref{ass::moment} with $p> \max \left\lbrace  \frac{8-8\gamma}{\gamma + \beta -1} , 2 \right\rbrace$. Suppose also that $ 2\gamma +2\nu > 3$, $\nu + \beta < 1$, $2\gamma -2\beta > 1$, $6\gamma +2\nu > 7$ and that $\gamma + \beta > 1$.  Then  
\[
A_{t} \xrightarrow[t\to + \infty]{a.s} \Sigma^{-1/2}  \quad \quad \text{ and } \quad \quad A_{t,\tau '} \xrightarrow[t\to + \infty]{a.s} \Sigma^{-1/2} 
\]
In addition, $\theta_{t}$ and $\theta_{t,\tau}$ defined by \eqref{thetaT::stream} and \eqref{thetaTAU::stream} satisfy
\[
\left\| \theta_{t} - \theta^{*} \right\|^{2} = O \left( \frac{\ln t}{t^{\nu}} \right) \quad a.s.  \quad \quad \text{ and } \quad \quad \left\| \theta_{t,\tau} - \theta^{*} \right\|^{2} = O \left( \frac{\ln t}{t^{\nu}} \right) \quad a.s.
\]
\end{theo}
The proof is given in Section \ref{proof}. 
Again, the restricted conditions on $\gamma , \nu , \beta$ are due to the use of Robbins-Siegmund Theorem and should be improved. 

\begin{theo}\label{theo::tlc::stream}
Suppose Assumptions \ref{ass1},  \ref{ass::sigma::lip}, \ref{ass::hess} and \ref{ass::hess::lip}    hold as well as \ref{ass::moment} for $p> \max \left\lbrace  \frac{8-8\gamma}{\gamma + \beta -1} , 2 \right\rbrace$. Suppose also that $ 2\gamma +2\nu > 3$, $\nu + \beta < 1$, $2\gamma -2\beta > 1$, $6\gamma +2\nu > 7$ and that $\gamma + \beta > 1$. Then    $\theta_{t,\tau}$ defined by   \eqref{thetaTAU::stream} satisfy
\[
\left\| \theta_{t,\tau} - \theta^{*} \right\|^{2} = O \left( \frac{\ln nt}{nt} \right) \quad a.s. \quad \quad \text{ and } \quad \quad   \sqrt{nt} \left( \theta_{t,\tau} - \theta^{*} \right) \xrightarrow[t \to + \infty]{\mathcal{L}} \mathcal{N} \left( 0 , H^{-1}\Sigma H^{-1} \right)
\]
with $\Sigma := \Sigma \left( \theta^{*} \right)$ and $H:= \nabla^{2}F \left( \theta^{*} \right)$.
\end{theo}
The proof is very similar to the one of Theorem \ref{theo::tlc} and is therefore not given.  Note that we ultimately obtain $\|\theta_{t,\tau} - \theta^{*} \|^2 = O\left(\frac{\ln N_t}{N_t}\right)$ a.s., which means the convergence rate is the same as the one of the WAFA algorithm and the estimates are still asymptotically efficient, but we drastically reduce the calculus time.

\section{Applications}\label{sec::simu}
In this section, we carry out some numerical experiments to investigate the performance of our proposed Full AdaGrad and Streaming Full AdaGrad  algorithms. 
Our investigation begins with the application of these algorithms to the linear regression model on simulated data. 
The choice of linear regression is strategic. 
Indeed, with this model we are able to obtain the exact values of the matrix $\Sigma = \Sigma(\theta^*)$, which allows us to also evaluate the performances of our estimates of $\Sigma^{-1/2}$. 
 {
We then extend our experiments
by applying our algorithms to the logistic regression framework
on seven real data sets
(see Table \ref{tab:datasets} for details).
We thereby test the adaptability of our methods in handling complex, real-life data.
}
Throughout these comparative experiments, we employ the AgaGrad algorithm defined in \eqref{DefAdaDiag} and its weighted averaged version as a benchmark. 
The Weighted Averaged AdaGrad (WAA) is formulated following the same principles as those outlined for $\theta_{t,\tau}$ in \eqref{thetaTAU}.

\subsection{Discussion about the hyper-parameters involved in the different algorithms}
Although in the previous sections, we imposed several restrictions on hyperparameters $\beta$, $\gamma$, and $\nu$ purely for technical reasons to derive the convergence rates of the algorithms theoretically, in our experiments, we simply set $\beta = \gamma = \nu = \frac{3}{4}$. 
We will demonstrate that such a choice of hyperparameters does not affect the practical performance of the algorithms. 
Furthermore, for Full AdaGrad, we choose $c_\beta = c_\gamma = c_\nu = 1$, 
but for Full AdaGrad Streaming, while $c_\beta$ and $c_\gamma$ are still set to 1, we set $c_\nu = \sqrt{n}$. 
Since Full AdaGrad Streaming updates $\theta_t$ only $\frac{1}{n}$ times as often as Full AdaGrad and AdaGrad , we increase the step size of each $\theta_t$ update in Full AdaGrad Streaming by choosing a larger $c_\nu$.
However, for the AdaGrad algorithm defined in \eqref{DefAdaDiag}, 
we set $\nu_t = t^{-1/4}$, since $\left\{\mathcal{G}_t\right\}^{-1/2}$ 
inherently converges a rate of $1/\sqrt{t}$.
For the Full AdaGrad algorithms, we always initialize $A_0$ as $0.1I_d$. 
Finally, we set $\tau,\tau'=2$ for all weighted averaged estimates. 

\subsection{Linear regression on simulated data}
We first perform experiments with simulated data, considering the linear regression model. 
Let $(X, Y)$ be a random vector taking values in $\mathbb{R}^d \times \mathbb{R}$. 
Consider the case where $X$ is a centered Gaussian random vector 
and 
$$Y = X^T \theta^* + \varepsilon,$$
where $\theta^*$ is a parameter of $\mathbb{R}^d$ 
and $\varepsilon \sim \mathcal{N}(0,1)$ is independent from $X$. 
If the matrix $\mathbb{E}[XX^{T}]$ is positive, $\theta^{*}$ is the unique minimizer of the function $F$ defined for all $h \in \mathbb{R}^p$ by
$$F(h) = \frac{1}{2} \mathbb{E}\left[(Y - h^T X)^2\right].$$
In the upcoming simulations,  we simulate i.i.d copies of $X \sim \mathcal{N}(0, \Sigma_X)$ for each sample, where $\Sigma_X$ is a positive definite covariance matrix given later. Note that in this case  the variance of the gradient satisfy $\Sigma =\Sigma_X$. 
 Parameter $\theta^*$ is randomly selected as a realization from a uniform distribution over the hypercube $[-2, 2]^{d}$. 
 We then estimate $\theta^*$ using the different algorithms and compare their performances.

\subsubsection{AdaGrad vs. Full AdaGrad}
We first compare the performance of Full AdaGrad, AdaGrad and their weighted averaged versions. 
We consider two different structures for $\Sigma_X$. 
The first one is $\Sigma_X = I_d$, 
leading to the case of independent   predictors.
The second one is $\Sigma_X = R$ with $R_{i,j} = 0.9^{\abs{i-j}}$, 
leading to strong correlation between predictors. 
To compare the two algorithms, we compute the mean-squared error of the distance from $\theta_t$ to $\theta^*$ by averaging over 100 samples. 
We initialize $\theta_0$ as $\theta_0=\theta^* + \frac{1}{2}E$, where $E \sim \mathcal{N}(0, I_d)$ for both algorithms.
Figures~\ref{graph::courbe} and \ref{graph::courbe2} depict the evolution of the mean squared error with respect to the sample size {for $d=200$ and $d=400$, respectively.}
\begin{figure}[h!]
	\centering
\includegraphics[width=0.7\textwidth]{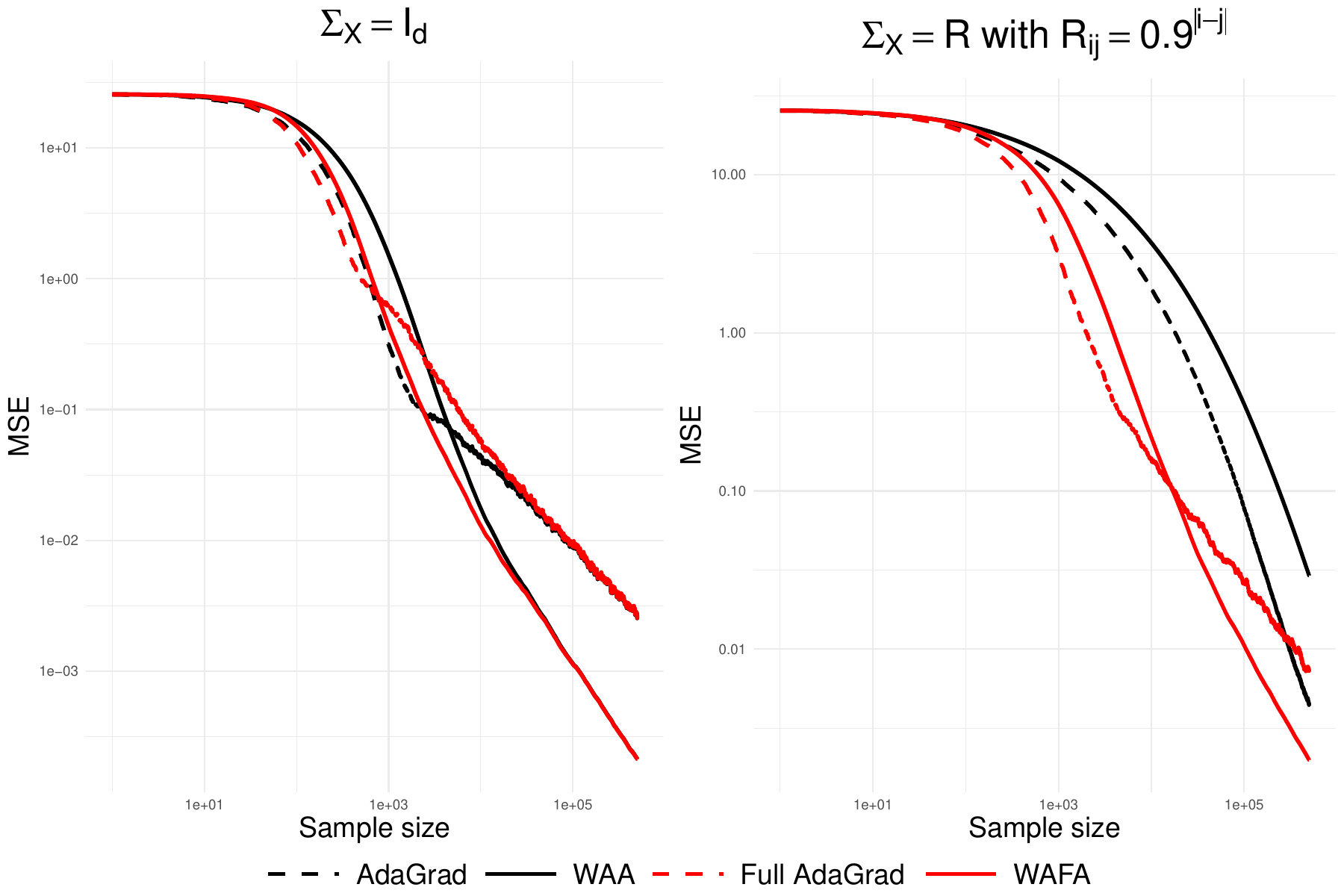}
\caption{{Linear regression case with $(N,d)=(500000,200)$. 
Mean squared error with respect to the sample size for AdaGrad and Full AdaGrad algorithms with their weighted averaged versions.
Two values  of $\Sigma_{X}$ are considered: $\Sigma_X = I_d$ (one the left) and $\Sigma_{X} = R$ (on the right).}}
\label{graph::courbe}
\end{figure}
\begin{figure}[h!]
	\centering
	\includegraphics[width=0.7\textwidth]{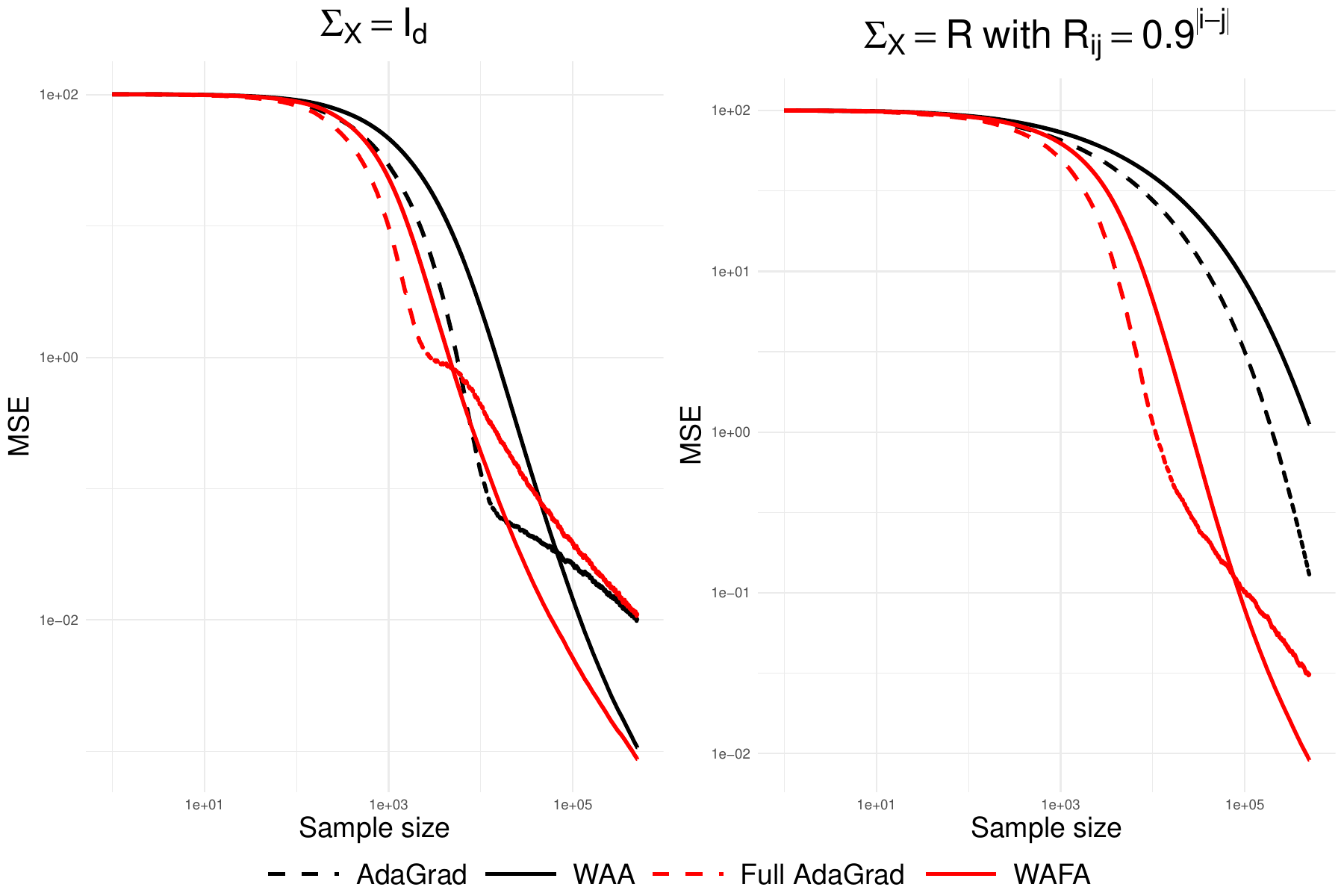}
	\caption{{Linear regression case with $(N,d)=(500000,400)$. 
		Mean squared error with respect to the sample size for AdaGrad and Full AdaGrad algorithms with their weighted averaged versions.
		Two values  of $\Sigma_{X}$ are considered: $\Sigma_X = I_d$ (one the left) and $\Sigma_{X} = R$ (on the right).}}
	\label{graph::courbe2}
\end{figure}
When $\Sigma_X$ is the identity matrix, AdaGrad and FullAdaGrad perform almost identically, and without surprise, the weighted averaged estimates enables to accelerate the convergence.
In this case, $\Sigma$ is a diagonal matrix, hence when AdaGrad only uses the diagonal elements, it does not lose any information. 
However, when there are strong correlations between predictors, 
as the off-diagonal elements of $\Sigma$ are no longer zero, Full AdaGrad significantly outperforms AdaGrad. 
{
In these cases, Full AdaGrad notably converges faster. 
From Figure \ref{graph::courbe2}, we can observe that as the dimensionality increases, Full AdaGrad’s advantage over standard AdaGrad becomes more pronounced. }
This highlights the significance of using Full AdaGrad over AdaGrad when addressing non-diagonal variance.

\subsubsection{Study of the full Adagrad streaming version.}
In this section, we demonstrate that the SWAFA can run in shorter time on the same dataset compared to WAFA, while achieving comparable results. 
We consider three different block sizes: $n=d$, $n=\sqrt{d}$, and $n=1$. 
Note that in the case $n=1$, SWAFA and WAFA  algorithms are the same.
We simulate the data in exactly the same manner as in the previous paragraph.
Through 100 samples, we plot the algorithm's running time, 
and the estimation error of $\theta$ given by $\norm{\theta_{t,\tau}-\theta^*}$ 
for the three different block sizes. 
Moreover, since we have the exact values of $\Sigma$, 
we  also evaluate the estimates of $\Sigma^{-1/2}$ 
by computing the error defined by
$\norm{A_{t,\tau}-\Sigma^{-1/2}}_{F}$.
\begin{figure}[H]
	\centering
	\includegraphics[width=0.7\textwidth]{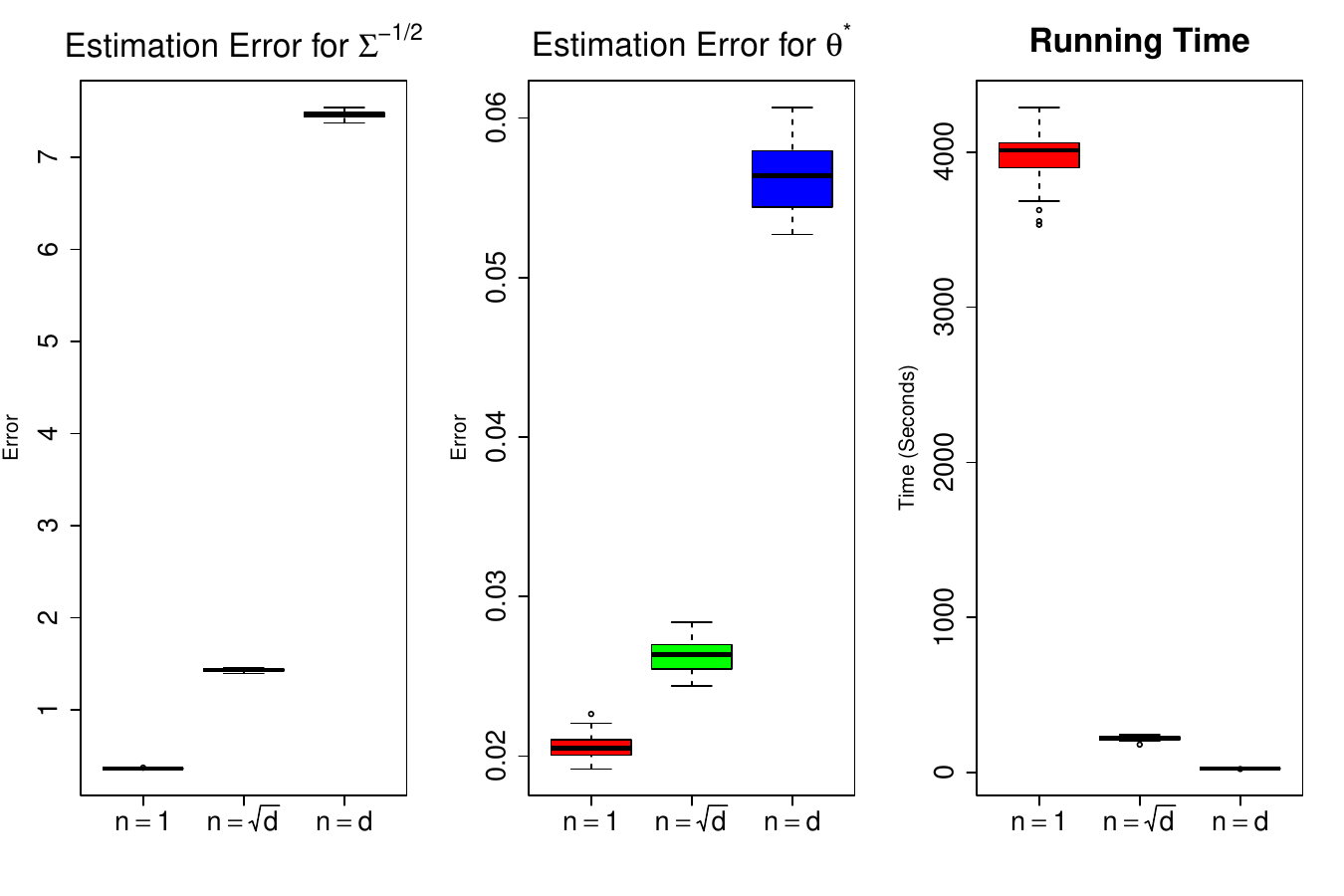}
	\caption{{From the left to the right:  boxplots  of the estimation errors for $\Sigma^{-1/2}$, boxplot of the  estimation errors for $\theta$ and 
boxplots of running time. In each case, $\Sigma_{X} = I_d$,  $(N,d)=(1000000,400)$ and three possible values of the streaming batch size are considered: $n=1,\sqrt{d},d$.}}
\label{graph::boxplot}
\end{figure}
\begin{figure}[H]
	\centering
	\includegraphics[width=0.7\textwidth]{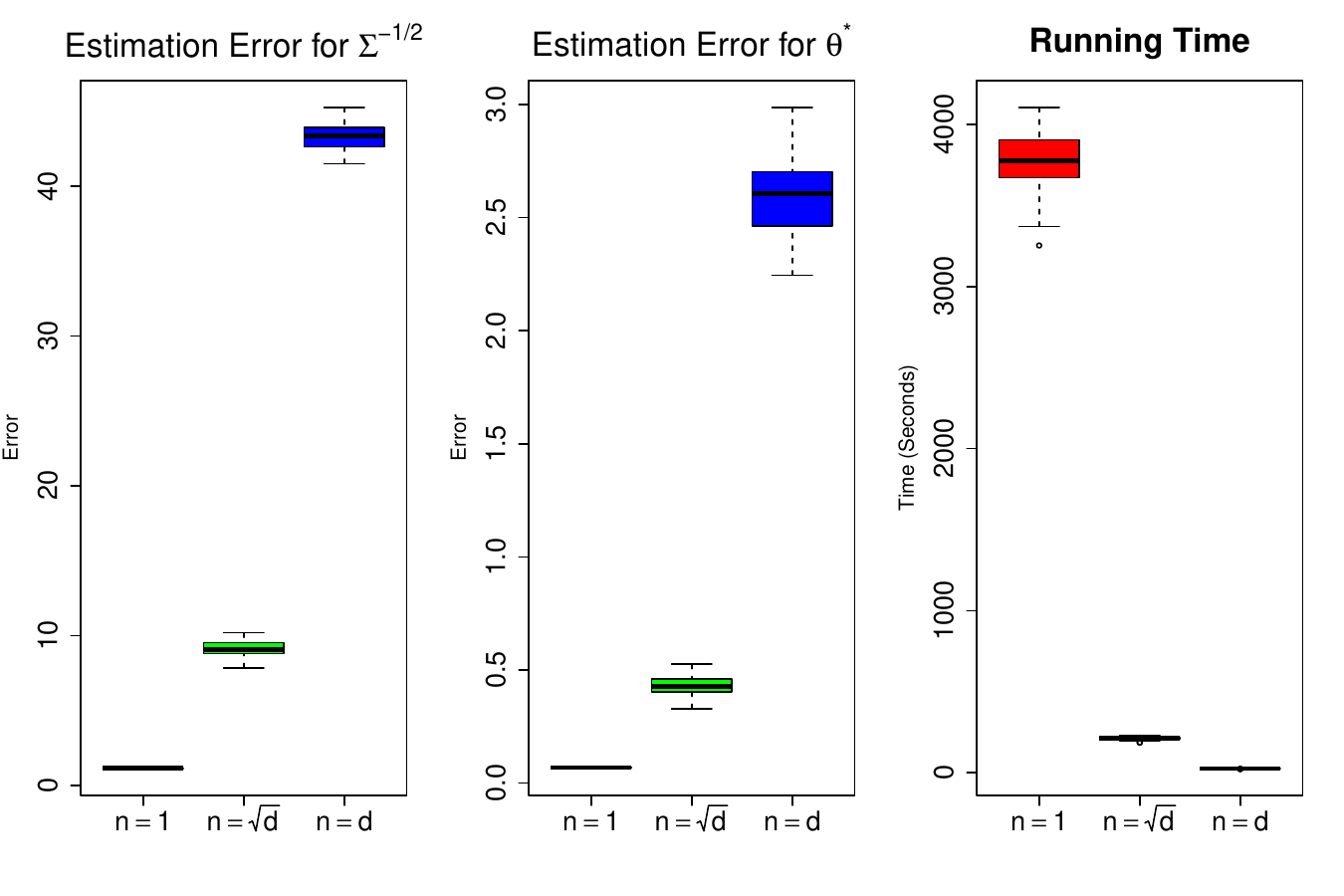}
	\caption{{From the left to the right:  boxplots  of the estimation errors for $\Sigma^{-1/2}$, boxplot of the  estimation errors for $\theta$ and 
		boxplots of running time. In each case, $\Sigma_{X} = R$, $(N,d)=(1000000,400)$ and three possible values of the streaming batch size are considered: $n=1,\sqrt{d},d$.}}
	\label{graph::boxplotId}
\end{figure}
We can see from Figures~\ref{graph::boxplot} and \ref{graph::boxplotId} 
that SWAFA significantly reduces computation time. 
Indeed, when $n = d$, 
most of the computation time is spent on reading the data and estimating the gradient. 
Considering $\Sigma_{X} = I_d$, SWAFA has a larger estimation error for 
$\Sigma^{-1/2}$ compared to WAFA, but it can still accurately estimate $\theta^*$. 
{
However, when we consider $\Sigma_{X} = R$, 
the estimation error of SWAFA with 
$n = d$ is much higher than that of WAFA, 
whereas SWAFA with 
$n = \sqrt{d}$ can strike a balance between running speed and estimation accuracy. 
In the face of ill-conditioned problems, choosing a smaller value of $n$ for SWAFA yields more stable results.
}

\subsection{Logistic regression on real data}
In this study, we evaluate several algorithms on a variety of real-world datasets. 
{
Table \ref{tab:datasets} outlines specifics about the data sets used in these experiments.
}
\begin{table}[H]
	\centering
	\label{tab:datasets}
	\begin{tabular}{lccc}
		\hline
		Dataset   & \# Features & Training Set Size & Testing Set Size \\
		\hline
		COVTYPE   & 54          & 290,506           & 290,505          \\
		MNIST     & 784         & 12,214            & 2,037           \\
		Mushrooms & 112          & 6,499             & 1,625           \\
		Adult     & 123          & 32,561            & 16,281           \\
		Epsilon   & 2000        & 400,000           & 100,000          \\
		Phishing  & 68          & 8,844               & 2,211              \\
		Web Page  & 300         & 3,470             & 46,279            \\
		\hline
	\end{tabular}
	\caption{{Description of Datasets}}
\end{table}
The \textbf{COVTYPE} dataset, originally collected by \cite{blackard1998comparison}, contains forest cover information. Note that COVTYPE includes multiple cover types. Our focus is on the most common forest cover type, “Spruce/Fir,” which accounts for about half of the dataset. We have simplified the “covertype” variable for our analysis by marking “Spruce/Fir” as 1 and all other types as 0. 
{
The \textbf{MNIST} dataset is a well-known dataset of handwritten digits \citep{lecun1998gradient}. It comprises 10 classes representing the handwritten digits 0--9, and we perform binary classification for digits 7--9. 
The \textbf{Mushrooms} dataset include various features of mushrooms for edibility classification \citep{lichman2013uci}. \textbf{Adult} contains census data used to predict income levels \citep{dua2017uci}, with eight categorical features one-hot encoded into binary features. 
\textbf{Epsilon} is part of the LIBSVM data collection \citep{chang2011libsvm}, which is a large-scale dataset for binary classification tasks. 
\textbf{Phishing} includes features of websites for detecting phishing attacks \citep{phishing_websites_327}, which contains only categorical features. We use binary encoding to generate feature vectors. 
\textbf{Web Page} is used for categorizing webpages, originally collected by \cite{platt199912} in his work on SVM-based text classification. 
These datasets are widely used to benchmark binary classification algorithms \citep{toulis2017asymptotic,juan2016field,yuan2012improved}. While MNIST, Adult, Epsilon, and Web Page come with predefined training and testing splits, we partition COVTYPE, Mushrooms, and Phishing equally into training and testing sets.
} 
The objective is to use logistic regression model for predicting the binary dependent variable of each dataset. 
We apply AdaGrad, Full AdaGrad, WAA, WAFA, SWAFA and SGD. 
For each algorithm, we initialize $\theta_0 = (0, \ldots,0)$. 
{
Table \ref{tab:acc}
presents for each data set and each algorithm the percentage of accuracy.
	SWAFA is split into two variants:  	$n=d$ and $n=\sqrt{d}$.
}

\begin{table}[H]
	\centering
	\label{tab:acc}
	\small
	\begin{tabular}{llccccccc}
		\hline
		\multirow{2}{*}{Data set} & \multirow{2}{*}{Set} & \multirow{2}{*}{AdaGrad} & \multirow{2}{*}{WAA} & \multirow{2}{*}{Full AdaGrad} & \multirow{2}{*}{WAFA} & \multicolumn{2}{c}{SWAFA} & \multirow{2}{*}{SGD} \\
		&                     &                          &                           &                     &                     & $n=d$ & $n=\sqrt{d}$ &                      \\
		\hline
		\multirow{2}{*}{COVTYPE}   & Train & 75.71 & 75.56 & 75.67 & 75.58 & 75.59 & 75.58 & 75.60  \\
		& Test  & 75.74 & 75.58 & 75.69 & 75.61 & 75.62 & 75.61 & 75.69  \\
		\hline
		\multirow{2}{*}{MNIST}     & Train & 95.77 & 95.60 & 95.34 & 95.86 & 92.92 & 95.64 & 95.04  \\
		& Test  & 95.88 & 95.97 & 95.68 & 95.88 & 93.57 & 95.83 & 95.14  \\
		\hline
		\multirow{2}{*}{Mushrooms} & Train & 99.60 & 98.89 & 99.15 & 98.54 & 97.75 & 98.48 & 97.54  \\
		& Test  & 99.75 & 99.02 & 99.32 & 98.83 & 98.34 & 98.83 & 98.15  \\
		\hline
		\multirow{2}{*}{Adult}     & Train & 84.76 & 84.75 & 84.78 & 84.69 & 84.42 & 84.76 & 83.98  \\
		& Test  & 85.09 & 85.09 & 85.14 & 85.19 & 84.79 & 85.17 & 84.08  \\
		\hline
		\multirow{2}{*}{Epsilon}   & Train & 87.18 & 86.54 & 85.98 & 84.81 & 77.58 & 83.16 & 71.79  \\
		& Test  & 87.05 & 86.40 & 85.86 & 84.69 & 77.46 & 82.99 & 71.66  \\
		\hline
		\multirow{2}{*}{Phishing}  & Train & 93.06 & 93.05 & 92.87 & 92.65 & 90.26 & 92.48 & 88.08  \\
		& Test  & 92.76 & 92.54 & 92.34 & 92.22 & 89.69 & 91.95 & 87.92  \\
		\hline
		\multirow{2}{*}{Web Page}  & Train & 97.00 & 97.00 & 96.95 & 97.12 & 96.92 & 96.92 & 96.92  \\
		& Test  & 97.04 & 97.05 & 97.04 & 97.14 & 97.04 & 97.04 & 97.04  \\
		\hline
	\end{tabular}
	\caption{{Train and Test accuracy (in \%)  	for AdaGrad, WAA, Full AdaGrad, WAFA, 
SWAFA ($n = d$ and $n = \sqrt{d}$) and SGD algorithmes.}}
\end{table}

{
results of Table \ref{tab:acc}
show that the proposed methods achieve higher classification rates compared to SGD, 
but the difference from AdaGrad is not significant. 
As discussed earlier, the advantage of Full AdaGrad is only evident when there is strong correlation among different dimensions of the gradient, which is not the case for these classical datasets.
} 
Overall, the results indicate that the proposed methods are applicable to real-world data.
 {
 Furthermore, when setting \( n = d \) for Streaming Full AdaGrad, we observe a slight drop in accuracy on some datasets where the dimension is too large, particularly on Epsilon. 
 However, when choosing \( n = \sqrt{d} \), the results remain more stable across all datasets. 
 Therefore, in practical applications, if the sample size is not significantly larger than the dimensionality, 
 selecting $ n = \sqrt{d}$ may be a more robust choice. 
 }

 \section*{Conclusion}
This work propose novel approaches to Full AdaGrad algorithms. The core innovation   lies in applying a Robbins-Monro type algorithm for estimating the inverse square root of the variance of the gradient. By proving the convergence rate of the proposed estimates, we lay a theoretical foundation that establishes the reliability of our approach. Through numerical studies, we have shown that our approach offers substantial advantages over traditional AdaGrad algorithms that rely solely on diagonal elements. Moreover, we introduce a streaming variant of our method, which further reduces computational complexity. We show that the streaming estimates are also asymptotically efficient. An extension of this work would be to understand the possible impact of the dimension of the behavior of the estimates, maybe through a non asymptotic theoretical study. 

 \section*{Aknowledgements} 
The authors would like to thank Guillaume Sallé for his assistance in improving the proofs.

\section{Proofs}\label{proof}
To simplify our notation, in the following we denote $\widehat{\Sigma}_t = g_t(\theta_{t-1,\tau})g_t(\theta_{t-1,\tau})^T$ ,
$W_t := g_{t+1}\pa{\theta_{t,\tau}}g_{t+1}\pa{\theta_{t,\tau}}^TA_t$ and  $Q_t =  A_t^{1/2}\widehat{\Sigma}_{t+1}A_t^{1/2}$ with $ \norm{Q_t}_F = g_{t+1}\pa{\theta_{t,\tau}}^TA_tg_{t+1}\pa{\theta_{t,\tau}} $. 
\subsection{Proof of Theorem \ref{theo::consistency}}\label{sec::proConsisTheta}
{
The goal is to apply Theorem 1 in \cite{GBW2023OND}. Notably, the proof of this theorem does not rely on any continuity assumption for the function $\Sigma$. It states that, under our assumptions, if we can control the eigenvalues of $A_{t,\tau'}$, specifically if  
\[
\lambda_{\max} \left( A_{t,\tau '} \right) = \mathcal{O} \left( t^{1-\gamma } \right) \quad a.s. \quad \text{and} \quad \lambda_{\min} \left( A_{t,\tau '} \right) = \mathcal{O} \left( \beta_{t} \right) \quad a.s.,
\]
then the estimates $\theta_{t}$ are strongly consistent, and consequently, so are $\theta_{t,\tau}$.  
 To apply this theorem, we first  control the eigenvalues of the random stepsequence $A_{t} $, and so, without requiring knowledge on the behavior of the estimate $\theta_{t}$.  In this aim let us recall that $A_{t}$ is defined recursively as follows:
 \begin{equation*}
	A_{t+1} = A_t - \gamma_{t+1}\pa{A_tg_{t+1}\pa{\theta_{t}}g_{t+1}\pa{\theta_{t}}^TA_t -Id}\mathbf{1}_{\acco{g_{t+1}\pa{\theta_{t}}^TA_tg_{t+1}\pa{\theta_{t}}\le\beta_{t+1}}},
\end{equation*}
}

\paragraph*{Study on the largest eigenvalue of $A_t$ and $A_{t,\tau'}$.} 
It is obvious that the matrix $A_t\widehat{\Sigma}_{t+1}A_t$ is positive semi-definite, so that
\begin{align*}
	\lambda_{\max}(A_{t+1}) \le \lambda_{\max}\pa{A_t + \gamma_{t+1}I_d\mathbf{1}_{\acco{\norm{Q_t}_F\le\beta_{t+1}}}} 
	\le \lambda_{\max}\pa{A_t + \gamma_{t+1}I_d}
\end{align*}
Therefore,

\begin{equation}
\label{maj::an}\lambda_{\max}(A_{t+1}) \le \lambda_{\max}\pa{A_0} +\sum_{k=0}^{t}\gamma_{k+1} = \mathcal{O}\pa{t^{1-\gamma}} , 
\end{equation}
and one can derive, since $A_{t,\tau '}$ is a convex combination of $A_{k}, k=0 , \ldots , t $, that $\lambda_{\max} \left( A_{t,\tau '} \right) = O \left( t^{1-\gamma}\right) $ a.s.

\paragraph*{Study on the smallest eigenvalue of $A_t$ and $A_{t,\tau '}$.}
We now provide an asymptotic bound of $\lambda_{\min}(A_t)^{-1}$, without necessitating knowledge on the behavior of the estimate $\theta_{t}$.  
Thanks to the truncation term ($\mathbf{1}_{\left\| Q_{t} \right\|_{F} \leq \beta_{t+1}}$), {the matrix $I_d - \gamma_{t+1}A_t^{1/2}\widehat{\Sigma}_{t+1}A_t^{1/2}\mathbf{1}_{\acco{\norm{Q_t}_F\le\beta_{t+1}}} $ is symmetric and positive, and one can easily verify that $A_t$ is positive  for all $t\geq 0$ (one can see equality \eqref{decompositionAt} to be convinced)}. We now give a better lower bound of its eigenvalues. First, remark that since $A_{t}$ is symmetric and positive, one can rewrite $A_{t+1}$  as
\begin{align}
	A_{t+1} &= A_t - \gamma_{t+1}A_t\widehat{\Sigma}_{t+1}A_t\mathbf{1}_{\acco{\norm{Q_t}_F\le\beta_{t+1}}} + \gamma_{t+1}I_d\mathbf{1}_{\acco{\norm{Q_t}_F\le\beta_{t+1}}} \nonumber\\
	&= A_t^{1/2}\pa{I_d - \gamma_{t+1}A_t^{1/2}\widehat{\Sigma}_{t+1}A_t^{1/2}\mathbf{1}_{\acco{\norm{Q_t}_F\le\beta_{t+1}}}}A_t^{1/2} + \gamma_{t+1}I_d\mathbf{1}_{\acco{\norm{Q_t}_F\le\beta_{t+1}}} \label{decompositionAt}
\end{align}
Note that by definition of $\widehat{\Sigma}_{t+1}$, the matrix $A_t^{1/2}\widehat{\Sigma}_{t+1}A_t^{1/2}$ is of rank $1$ and $$\norm{A_t^{1/2}\widehat{\Sigma}_{t+1}A_t^{1/2}}_{op} = \norm{A_t^{1/2}\widehat{\Sigma}_{t+1}A_t^{1/2}}_F = \norm{Q_t}_F.$$
Thus, since for all symmetric and positive matrices $A,B$ one has $\lambda_{\min}(ABA) \geq \lambda_{\min}\left( A^{2} \right) \lambda_{\min}(B)$, 
\begin{align*}
	\lambda_{\min}(A_{t+1}) &\ge \lambda_{\min}\pa{A_t^{1/2}\pa{I_d - \gamma_{t+1}A_t^{1/2}\widehat{\Sigma}_{t+1}A_t^{1/2}\mathbf{1}_{\acco{\norm{Q_t}_F\le\beta_{t+1}}}}A_t^{1/2}} + \gamma_{t+1}\mathbf{1}_{\acco{\norm{Q_t}_F\le\beta_{t+1}}}\\
	&\ge \lambda_{\min}(A_t)\pa{1-\gamma_{t+1}\norm{A_t^{1/2}\widehat{\Sigma}_{t+1}A_t^{1/2}}_{op}\mathbf{1}_{\acco{\norm{Q_t}_F\le\beta_{t+1}}}} + \gamma_{t+1}\mathbf{1}_{\acco{\norm{Q_t}_F\le\beta_{t+1}}}\\
	&\ge \lambda_{\min}(A_t)\pa{1-\gamma_{t+1}\beta_{t+1}} + \gamma_{t+1}\mathbf{1}_{\acco{\norm{Q_t}_F\le\beta_{t+1}}} .
\end{align*}
Let us now prove by induction that $\lambda_{\min} \left( A_{t} \right) \geq \frac{\lambda_{0}}{\beta_{t+1}}$ where $\lambda_{0} := \min \left\lbrace 1 , \lambda_{\min} \left( A_{0} \right) \beta_{1}^{-1} \right\rbrace$. By definition of $\lambda_{0}$, the property is clearly satisfied for $t=0$ and we suppose that is now the case for $t \geq 0$, i.e that $\lambda_{\min}(A_{t})>\frac{1}{\beta_{t+1}}$. Then, if $\norm{Q_t}_F>\beta_{t+1}$, one has
$$\lambda_{\min}(A_{t+1}) = \lambda_{\min}(A_{t}) \geq \frac{1}{\beta_{t+1}}>\frac{1}{\beta_{t+2}}.$$
If $\norm{Q_t}_F\le\beta_{t+1}$, one has
\begin{align*}
	\lambda_{\min}(A_{t+1}) &\ge \lambda_{\min}(A_t)\pa{1-\gamma_{t+1}\beta_{t+1}} + \gamma_{t+1}\\
	&\ge \frac{\lambda_{0}}{\beta_{t+1}}\pa{1-\gamma_{t+1}\beta_{t+1}} + \gamma_{t+1}\\
	&\geq \frac{\lambda_{0}}{\beta_{t+1}},
\end{align*}
where the last inequality comes from the fact that $\lambda_{0} \leq 1$. {Then, 
\[
\lambda_{\max} \left( A_{t}^{-1} \right) = O( \beta_{t}) \quad a.s.
\]
and since $A_{t,\tau'}$ is a convex combination of $A_{k}$, $k= 0 , \ldots , t$, it comes} 
\begin{equation}\label{min::lambda}
\lambda_{\max} \left( A_{t,\tau '}^{-1} \right) = O( \beta_{t}) \quad a.s.
\end{equation}
Then, applying  Theorem 1 in \cite{GBW2023OND}, it comes that $\theta_{t}$ and $\theta_{t,\tau}$ converge almost surely to $\theta^{*}$.

\subsection{Proof of Theorem \ref{theo::rate}}
Let $(D_t)$ be a sequence defined by $D_t := A_{t}\Sigma_{t-1}A_{t} - I_d$ where $\Sigma_t := \Sigma\pa{\theta_{t,\tau}}$.
By definition of $A_{t+1}$, we have
\begin{align*}
	D_{t+1}&=\pa{A_t-\gamma_{t+1}\pa{A_tW_t-I_d}\mathbf{1}_{\acco{\norm{Q_t}_F\le\beta_{t+1}}}}\Sigma_{t}\pa{A_t-\gamma_{t+1}\pa{A_tW_t-I_d}\mathbf{1}_{\acco{\norm{Q_t}_F\le\beta_{t+1}}}}  - I_d\\
&=A_t\Sigma_{t}A_t - I_d -\gamma_{t+1}\pa{\pa{A_tW_t-I_d}\Sigma_{t} A_t  + A_t\Sigma_{t} \pa{A_tW_t-I_d}}\mathbf{1}_{\acco{\norm{Q_t}_F\le\beta_{t+1}}}\\
&\quad+\gamma_{t+1}^2A_tW_t\Sigma_{t} A_tW_t\mathbf{1}_{\acco{\norm{Q_t}_F\le\beta_{t+1}}}-\gamma^2_{t+1}\pa{A_tW_t-I_d}\Sigma_{t}\mathbf{1}_{\acco{\norm{Q_t}_F\le\beta_{t+1}}}\\
&\quad - \gamma^2_{t+1}\Sigma_{t}\pa{A_tW_t-I_d}\mathbf{1}_{\acco{\norm{Q_t}_F\le\beta_{t+1}}} + \gamma^2_{t+1}\Sigma_{t}\mathbf{1}_{\acco{\norm{Q_t}_F\le\beta_{t+1}}}.
\end{align*}
In all the sequel let us denote 
\begin{align*}
R_{t} & =  \left\|  \gamma_{t+1}^2A_tW_t\Sigma_{t} A_tW_t\mathbf{1}_{\acco{\norm{Q_t}_F\le\beta_{t+1}}}\right\|_{F} + \left\| \gamma^2_{t+1}\pa{A_tW_t-I_d}\Sigma_{t}\mathbf{1}_{\acco{\norm{Q_t}_F\le\beta_{t+1}}} \right\|_{F}\\
&   + \left\| \gamma^2_{t+1}\Sigma_{t}\pa{A_tW_t-I_d}\mathbf{1}_{\acco{\norm{Q_t}_F\le\beta_{t+1}}}\right\|_{F} + \left\|  \gamma^2_{t+1}\Sigma_{t}\mathbf{1}_{\acco{\norm{Q_t}_F\le\beta_{t+1}}} \right\|_{F} .
\end{align*}
Then, using Cauchy-Schwarz inequality, it comes
\begin{align*}
\left\| D_{t+1} \right\|_{F}^{2} & \leq  \left\| A_{t}\Sigma_{t}A_{t} - I_{d} \right\|_{F}^{2} +  R_{t}^{2} + 2 R_{t}    \left\| A_{t}\Sigma_{t}A_{t} - I_{d} \right\|_{F}  \\
& + 2 R_{t}\gamma_{t+1} \left\| \pa{A_tW_t-I_d}\Sigma_{t} A_t -A_t\Sigma_{t} \pa{A_tW_t-I_d}\right\|_{F}\mathbf{1}_{\acco{\norm{Q_t}_F\le\beta_{t+1}}} \\
&  -4 \gamma_{t+1} \left\langle \pa{A_tW_t-I_d}\Sigma_{t} A_t \mathbf{1}_{\acco{\norm{Q_t}_F\le\beta_{t+1}}}, A_{t}\Sigma_{t}A_{t} - I_{d} \right\rangle_{F} \\
& + 2 \gamma_{t+1}^{2}\left\| \pa{A_tW_t-I_d}\Sigma_{t} A_t + A_t\Sigma_{t} \pa{A_tW_t-I_d}\right\|_{F}^{2}\mathbf{1}_{\acco{\norm{Q_t}_F\le\beta_{t+1}}} .
\end{align*} 
Then, applying inequality $ab \leq \frac{1}{2c}a^{2} + \frac{c}{2}b^{2}$ (with $a,b,c > 0$), it comes
\begin{align*}
2 R_{t}    \left\| A_{t}\Sigma_{t}A_{t} - I_{d} \right\|_{F} \leq \frac{1}{\gamma_{t+1}^{2}\beta_{t+1}^{2}} R_{t}^{2} + \gamma_{t+1}^{2}\beta_{t+1}^{2}  \left\| A_{t}\Sigma_{t}A_{t} - I_{d} \right\|_{F}^{2}
\end{align*}
and
\begin{align*}
2 R_{t} & \gamma_{t+1} \left\| \pa{A_tW_t-I_d}\Sigma_{t} A_t -A_t\Sigma_{t} \pa{A_tW_t-I_d}\right\|_{F}\mathbf{1}_{\acco{\norm{Q_t}_F\le\beta_{t+1}}} \\
 & \leq R_{t}^{2} + \gamma_{t+1}^{2} \left\| \pa{A_tW_t-I_d}\Sigma_{t} A_t +A_t\Sigma_{t} \pa{A_tW_t-I_d}\right\|_{F}^{2}\mathbf{1}_{\acco{\norm{Q_t}_F\le\beta_{t+1}}} .
\end{align*}
Then
\begin{align*}
\left\| D_{t+1} \right\|_{F}^{2} & \leq \left( 1+   \gamma_{t+1}^{2}\beta_{t+1}^{2} \right) \left\| A_{t}\Sigma_{t}A_{t} - I_{d} \right\|_{F}^{2} + \left( 2 + \frac{1}{\gamma_{t+1}^{2}\beta_{t+1}^{2}} \right) R_{t}^{2} \\
&  -4 \gamma_{t+1} \left\langle \pa{A_tW_t-I_d}\Sigma_{t} A_t \mathbf{1}_{\acco{\norm{Q_t}_F\le\beta_{t+1}}}, A_{t}\Sigma_{t}A_{t} - I_{d} \right\rangle_{F} \\
& + 2 \gamma_{t+1}^{2}\left\| \pa{A_tW_t-I_d}\Sigma_{t} A_t + A_t\Sigma_{t} \pa{A_tW_t-I_d}\right\|_{F}^{2}\mathbf{1}_{\acco{\norm{Q_t}_F\le\beta_{t+1}}}
\end{align*} 
The aim now, considering the filtration $\left( \mathcal{F}_{t} \right)$ defined for all $t \geq 1$ by $\mathcal{F}_{t} = \sigma \left( X_{1} , \ldots , X_{t} \right)$,  is to use Robbins-Siegmund theorem \citep{robbins1971convergence}. To do so, we have to upper bound the conditional expectation of each term on the right-hand side of previous inequality. First,the following lemma which  gives the rate of convergence of  $\mathbb{E} \left[ R_{t}^{2} |\mathcal{F}_{t} \right]$.
\begin{lem}\label{lem0}
Under the same assumptions as in Theorem \ref{theo::rate} there are random positive sequences $R_{1,t}$ and $R_{t}^{*}$
satisfying
\[
\sum_{t \geq 1 } R_{1,t} < + \infty \quad a.s. \quad \quad \text{and} \quad \quad 
\sum_{t \geq 1 } R_{t}^{*} < + \infty \quad a.s.
\]
such that
\begin{align*}
\mathbb{E} \left[ R_{t}^{2} |\mathcal{F}_{t} \right]\leq \left( 1+ R_{1,t} \right) \left\| D_{t} \right\|_{F}^{2} + R_{t}^{*}.
\end{align*}

\end{lem}
The proof is given in Section \ref{sec::proof}.
Furthermore, one has to give the link between $ \norm{A_{t}\Sigma_{t}A_{t} - I_{d}}_F^2$ and $\norm{D_{t}}_F^2$. This is given by the following lemma.

\begin{lem}\label{lemlink}
Under the same assumptions as in Theorem \ref{theo::rate} there are random positive sequences $R_{3,t}$ and $R_{4,t}$
satisfying
\[
\sum_{t \geq 1 } R_{3,t} < + \infty \quad a.s. \quad \quad \text{and} \quad \quad 
\sum_{t \geq 1 } R_{4,t} < + \infty \quad a.s.
\]
such that

\begin{align*}
\left\| A_{t}\Sigma_{t}A_{t} - I_{d} \right\|_{F}^{2} \leq \left( 1+ R_{3,t} \right) \left\| D_{t} \right\|_{F}^{2} + R_{4,t}.
\end{align*}

\end{lem}
The proof is given in Section \ref{sec::lemlink}.
Finally, let us   bound 
\[
(*):= -2 \gamma_{t+1} \left\langle \pa{A_tW_t-I_d}\Sigma_{t} A_t  \mathbf{1}_{\acco{\norm{Q_t}_F\le\beta_{t+1}}}, A_{t}\Sigma_{t}A_{t} - I_{d} \right\rangle_{F}. 
\]
First, note that
\begin{align*}
(*)  = -\underbrace{2 \gamma_{t+1} \left\langle \pa{A_tW_t-I_d}\Sigma_{t} A_t  , A_{t}\Sigma_{t}A_{t} - I_{d} \right\rangle_{F}}_{=: K_{1,t}} \\ \quad+ \underbrace{ 2 \gamma_{t+1} \left\langle \pa{A_tW_t-I_d}\Sigma_{t} A_t , A_{t}\Sigma_{t}A_{t} - I_{d} \right\rangle_{F} \mathbf{1}_{\acco{\norm{Q_t}_F > \beta_{t+1}}}}_{=: K_{2,t}}.
\end{align*}
We now bound each term on the right-hand side of previous equality.
The following Lemma gives the rate of convergence of $K_{2,t}$.
\begin{lem}\label{lemk2}
Under the same assumptions as in Theorem \ref{theo::rate},
\[
\sum_{t\geq 0} \mathbb{E} \left[  K_{2,t} |\mathcal{F}_{t} \right] <  + \infty \quad a.s.
\]
\end{lem}
The proof is given in Section \ref{sec::lemk2}.
Let us now focus on the {positivity of $\mathbb{E}\left[  K_{1,t} |\mathcal{F}_{t} \right]$.}
Let us denote $\widetilde{D}_t :=A_t\Sigma_tA_t-I_d$ and remark  that $\mathbb{E} \left[  K_{1,t} |\mathcal{F}_{t} \right]= 2\gamma_{t+1}\langle A_t\Sigma_t\widetilde{D}_t\; , \; \widetilde{D}_t \rangle_F.$ One has, since $\tilde{D}_{t}$ and $A_{t}\Sigma_{t}A_{t}$ commute and since $\tilde{D}_{t}$ is symmetric,
\begin{align*}
	\langle A_t\Sigma_t \widetilde{D}_t \; , \; \widetilde{D}_t \rangle_F = \text{tr}\pa{A_t\Sigma_t \widetilde{D}_t^2} =\text{tr}\pa{A_t\Sigma_t A_t A_t^{-1}\widetilde{D}_t^2}
	= \text{tr}\pa{A_t^{-1}A_t\Sigma_t A_t\widetilde{D}_t^2}.
\end{align*}
In a same way,
\begin{align*}
	\langle A_t\Sigma_t \widetilde{D}_t \; , \; \widetilde{D}_t \rangle_F &= \text{tr}\pa{A_t^{-1}(A_t\Sigma_t A_t\widetilde{D}_t)\widetilde{D}_t}   \\
	& =\text{tr}\pa{A_t^{-1} \widetilde{D}_t A_t\Sigma_t A_t \widetilde{D}_t} \\
	& =\text{tr}\pa{(A_t\Sigma_t A_t)(\widetilde{D}_tA_t^{-1}\widetilde{D}_t)}.
\end{align*}
Both $A_t\Sigma_t A_t$ and $\widetilde{D}_tA_t^{-1}\widetilde{D}_t$ are positive symmetric matrix, so that 
$$\langle A_t\Sigma_t \widetilde{D}_t \; , \; \widetilde{D}_t \rangle_F = \text{tr}\pa{(A_t\Sigma_t A_t)^{1/2}(\widetilde{D}_tA_t^{-1}\widetilde{D}_t)(A_t\Sigma_t A_t)^{1/2}} \ge 0.$$ 
Therefore, $K_{1,t}\ge0$ for all $t\ge0$.

\paragraph*{Upper bound of $\mathbb{E}\left[ \left\| D_{t+1} \right\|_{F}^{2} |\mathcal{F}_{t} \right]$ and first conclusions. } Resuming all previous bounds, one has
\begin{equation}\label{updt}
\mathbb{E}\left[ \left\| D_{t+1} \right\|_{F}^{2} |\mathcal{F}_{t} \right] \leq  \left( 1+ S_{1,t} \right) \left\| D_{t} \right\|_{F}^{2} + S_{2,t} - \mathbb{E}\left[  K_{1,t} |\mathcal{F}_{t} \right]
\end{equation}
with $K_{1,t}$ positive and
\[
\sum_{t \geq 1} S_{1,t} < + \infty \quad a.s \quad \quad \text{ and} \quad \quad \sum_{t \geq 1} S_{2,t} < + \infty \quad a.s.
\]
Then, applying Robbins-Siemund Theorem, $\left\| D_{t} \right\|_{F}^{2}:= \left\| A_{t}\Sigma_{t-1}A_{t} - I_{d} \right\|_{F}^{2}$ converges almost surely to a finite random variable. Observe that since $\Sigma_{t}$ converges almost surely to $\Sigma$ which is positive, this leads to 
\begin{equation}
\label{maj::At::O1} \lambda_{\max} \left( A_{t} \right) = O (1) \quad a.s.
\end{equation}
In addition, Robbins-Siegmund Theorem ensures that 
\[
\sum_{t \geq 1} \mathbb{E}\left[  K_{1,t} |\mathcal{F}_{t} \right] < + \infty \quad a.s.
\]
Remark that since for all symmetric positive matrices $A,B$, $\text{tr}(ABA) \geq \lambda_{\min}(A)^{2} Tr (B) \geq \lambda_{\min}(A)^{2}\| B\|_{F}^{2}$, and using this inequality three times,
\begin{align*}
\mathbb{E}\left[  K_{1,t} |\mathcal{F}_{t} \right]&= 2\gamma_{t+1} \text{tr}\pa{(A_t\Sigma_t A_t)^{1/2}(\widetilde{D}_tA_t^{-1}\widetilde{D}_t)(A_t\Sigma_t A_t)^{1/2}} \\ &  \geq 2 \gamma_{t+1}  \lambda_{\min} \left(  \left( A_{t}\Sigma_{t}A_{t} \right)^{1/2} \right)^{2} \text{tr} \left(  \widetilde{D}_tA_t^{-1}\widetilde{D}_t  \right) \\
&   \geq 2 \gamma_{t+1}  \lambda_{\min} \left(    A_{t} \right)^{2} \| \Sigma_{t}\|_{F}    \lambda_{\min} \left(  A_t^{-1} \right)\left\| \widetilde{D}_t^{2} \right\|_{F}  \\
\end{align*}
Then, 
\begin{align*}
\mathbb{E}\left[  K_{1,t} |\mathcal{F}_{t} \right]&    \geq 2 \gamma_{t+1} \frac{\lambda_{\min}\left( A_{t} \right)^{2}}{\lambda_{\max}\left( A_{t} \right)}   \left\| \Sigma_{t} \right\|_{F}\left\| A_{t} \Sigma_{t} A_{t} - I_{d} \right\|_{F}^{2} .
\end{align*}
Then, in order to conclude, one has to obtain a better lower bound of the smallest eigenvalue of $A_{t}$. This is given by the following lemma.

\begin{lem}\label{lem::eig0}
Under the same assumptions as in Theorem \ref{theo::rate}
$$\frac{1}{\lambda_{\min}(A_t)} = \mathcal{O}(\beta'_{t})= \mathcal{O}(t^{\frac{1-\gamma}{4}})\quad a.s.$$ 
\end{lem}
The proof is given in Section \ref{sec::lem::eig0}.
Lemma \ref{lem::eig0} means that $\liminf \lambda_{\min}\left( A_{t} \right) t^{\frac{1-\gamma}{4}} > 0$ a.s so that
\begin{equation*}
	\sum_{t\ge 1}\gamma_{t+1}\lambda^2_{\min}(A_t) =+\infty\quad a.s. 
	\end{equation*}
and since $\lambda_{\max} \left( A_{t} \right) =O (1)$ a.s., it comes
\begin{equation*}
\label{sommeinfinie}	\sum_{t\ge 1}\gamma_{t+1}\frac{\lambda^2_{\min}(A_t)}{\lambda_{\max} \left( A_{t} \right)} =+\infty\quad a.s. 
\end{equation*}

\paragraph*{Conclusion 1} Observe that
\begin{align*}
\mathbb{E}\left[  K_{1,t} |\mathcal{F}_{t} \right]&   \geq 2\gamma_{t+1}\frac{\lambda_{\min}\left( A_{t} \right)^{2}}{\lambda_{\max}\left( A_{t} \right)}   \left\| \Sigma_{t} \right\|_{F}\left\| A_{t} \Sigma_{t} A_{t} - I_{d} \right\|_{F}^{2}  \geq \underbrace{ \gamma_{t+1}\frac{\lambda_{\min}\left( A_{t} \right)^{2}}{\lambda_{\max}\left( A_{t} \right)} \left\| \Sigma_{t} \right\|_{F} \left\| D_{t} \right\|_{F}^{2}}_{\tilde{K}_{1,t}} 
\\ &\quad-\underbrace{4 \gamma_{t+1} \frac{\lambda_{\min}\left( A_{t} \right)^{2}}{\lambda_{\max}\left( A_{t} \right)} \left\| \Sigma_{t} \right\|_{F} \left\| A_{t} \left( \Sigma_{t} - \Sigma_{t-1} \right) A_{t} \right\|_{F}^{2}}_{=: R_{6,t}}
\end{align*}
and one can remark that
\[
\frac{\lambda_{\max}\left( A_{t} \right)}{\lambda_{\min}\left( A_{t} \right)^{2}} R_{6,t} \leq 4\gamma_{t+1}  \left\| A_{t} \right\|_{F}^{5} \left\| \Sigma_{t} \right\|_{F} \frac{\ln t^{2\tau}}{\left( \sum_{k=0}^{t} \ln (k+1)^{\tau} \right)^{2}}L_{\Sigma}^{2} \left\| \theta_{t} - \theta_{t-1,\tau} \right\|^{2} = o \left( \frac{1}{t^{2}} \right) \quad a.s.
\]
i.e $\sum_{t \geq 1}R_{6,t} < + \infty$ a.s. and rewriting 
\[
\mathbb{E} \left[ \left\|  D_{t+1} \right\|_{F}^{2} |\mathcal{F}_{t} \right] \leq \left( 1+ S_{1,t} \right) \left\| D_{t} \right\|_{F}^{2} + S_{2,t} + R_{6,t} - \tilde{K}_{1,t}
\]
and applying Robbins-Siegmund Theorem, it comes
\[
\sum_{t \geq 1}\gamma_{t+1}\frac{\lambda_{\min}\left( A_{t} \right)^{2}}{\lambda_{\max}\left( A_{t} \right)} \left\| \Sigma_{t} \right\|_{F} \left\| D_{t} \right\|_{F}^{2} < + \infty \quad a.s.
\]
Then, equality \eqref{maj::At::O1} implies that $\liminf \left\| D_{t} \right\|_{F}^{2} = 0$ a.s, so that, since  $\left\| D_{t} \right\|_{F}^{2}$ converges almost surely to a finite random variable, $\left\| D_{t} \right\|_{F}^{2}$ converges almost surely to $0$, i.e
\[
A_{t}\Sigma_{t-1}A_{t} - I_{d} \xrightarrow[t\to + \infty]{a.s} 0
\]
and since $\Sigma_{t-1}$ converges almost surely to $\Sigma$, 
\[
A_{t} \xrightarrow[t\to + \infty]{a.s} \Sigma^{-1/2}.
\]

\paragraph*{Conclusion 2} Theorem 2 in \cite{GBW2023OND} states that under our assumptions, as soon as $A_{t,\tau'}$ converges almost surely to a positive matrix (which is satisfied), we have the desired rate of convergence, i.e we have
\[
\left\| \theta_{t} - \theta^{*} \right\|^{2} = O \left( \frac{\ln t}{t^{\nu}} \right) \quad a.s.
\]

\subsection{Proof of Theorem \ref{theo::tlc}}
The aim is to apply Theorem 4 in \cite{GBW2023OND}. This theorem states that under our assumptions, if there is $\delta > 0$ such that
\[
\frac{1}{\sum_{k=0}^{t} \ln (k+1)^{\tau}} \sum_{k=0}^{t} \ln (k+1)^{\tau +1/2 + \delta} \left\| A_{k+1,\tau '}^{-1} - A_{k,\tau '}^{-1} \right\|_{op} (k+1)^{\gamma /2} = O \left( \frac{1}{t^{\nu '}} \right) \quad a.s.
\]
for some $\nu ' > 1/2$, then we have the desired result.
Here, observe that 
\begin{align*}
\left\| A_{k+1,\tau '}^{-1} - A_{k,\tau '}^{-1} \right\|_{op} &  \leq \left\| A_{k+1,\tau '}^{-1} \right\| \left\| A_{k,\tau '}^{-1} \right\|_{op} \left\| A_{k+1,\tau '} - A_{k,\tau '} \right\|_{op}  \\
& \leq \left\| A_{k+1,\tau '}^{-1} \right\| \left\| A_{k,\tau '}^{-1} \right\|_{op} \frac{\ln (t+1)^{\tau '}}{\sum_{k=0}^{t} \ln (k+1)^{\tau '}} \left\| A_{k+1} - A_{k,\tau '} \right\| . 
\end{align*}
Since $A_{t}$ and $A_{t,\tau '}$ converge  almost surely to the positive matrix $\Sigma^{-1/2}$, it comes that 
\[
\left\| A_{k+1,\tau '}^{-1} - A_{k,\tau '}^{-1} \right\|_{op}  = o \left( \frac{1}{t} \right) \quad a.s.
\]
which concludes the proof since $\gamma < 1$.

\subsection{Proofs of Lemma \ref{lem0}}\label{sec::proof}
The aim is then to give an upper bound of the four terms composing $R_{t}$. This is given by the following lemma.

\begin{lem}\label{lem1}
Under the same assumtions as in Theorem \ref{theo::rate}, one has
\[
\sum_{t\geq 0} \frac{1}{\gamma_{t+1}^{2}\beta_{t+1}^{2}}\mathbb{E}\left[ \norm{\gamma_{t+1}^2A_tW_t\Sigma_{t} A_tW_t\mathbf{1}_{\acco{\norm{Q_t}_F\le\beta_{t+1}}}}^2_F |\mathcal{F}_{t} \right] < + \infty \quad a.s.
\]
\end{lem}

\begin{lem}\label{lem2}
Under the same assumptions as in Theorem \ref{theo::rate} there are random positive sequences $R_{1,t}$ and $R_{t}^{*}$
satisfying
\[
\sum_{t \geq 1 } R_{1,t} < + \infty \quad a.s. \quad \quad \text{and} \quad \quad 
\sum_{t \geq 1 } R_{2,t}  < + \infty \quad a.s.
\]
such that
\begin{align*}
\frac{1}{\gamma_{t+1}^{2}\beta_{t+1}^{2}}\mathbb{E}\cro{ \left\| \gamma^2_{t+1}\Sigma_{t}\pa{A_tW_t-I_d}\mathbf{1}_{\acco{\norm{Q_t}_F\le\beta_{t+1}}}\right\|_{F}^{2}| \mathcal{F}_t} \leq \left( 1+ R_{1,t} \right) \left\| D_{t} \right\|_{F}^{2} + R_{2,t}.
\end{align*}

\end{lem}
\begin{lem}\label{lem3}
Under the same assumptions as in Theorem \ref{theo::rate},
\[
\sum_{t \geq 0} \frac{1}{\gamma_{t+1}^{2}\beta_{t+1}^{2}}\mathbb{E}\cro{  \left\|  \gamma^2_{t+1}\Sigma_{t}\mathbf{1}_{\acco{\norm{Q_t}_F\le\beta_{t+1}}} \right\|_{F}^{2}| \mathcal{F}_t} < + \infty \quad a.s.
\]
\end{lem}

The proofs of Lemma \ref{lem1} and \Ref{lem2} are given thereafter and the one of Lemma \ref{lem3} is quite straightforward. Observe that we juste give the upper bound for $\gamma_{t+1}^{-2}\beta_{t+1}^{-2}\mathbb{E}\left[ R_{t} |\mathcal{F}_{t} \right]$ since the upper bound of $2\mathbb{E}\left[ R_{t} |\mathcal{F}_{t} \right]$ is direct application of these last ones (and we have $\gamma_{t+1}^{-2}\beta_{t+1}^{-2} \xrightarrow[t\to + \infty]{} + \infty$).

\begin{proof}[Proof of Lemma \ref{lem1}]
Thanks to Assumption \ref{ass::moment}, since $\| AB \|_{F} \leq \| A \|_{F} \| B\|_{F}$, we have
\begin{align*}
\mathbb{E} \left[ 	\norm{A_tW_t}_F^4\mathbf{1}_{\acco{\norm{Q_t}_F\le\beta_{t+1}}} |\mathcal{F}_{t} \right] &\le \mathbb{E} \left[  \norm{g_{t+1}\pa{\theta_{t,\tau}}}^4\norm{A_tg_{t+1}\pa{\theta_{t,\tau}}}^4\mathbf{1}_{\acco{\norm{Q_t}_F\le\beta_{t+1}}}|\mathcal{F}_{t} \right] \\
	&\le \mathbb{E} \left[  \norm{g_{t+1}\pa{\theta_{t,\tau}}}^4\norm{A_t^{1/2}Q_tA_t^{1/2}}_F^2\mathbf{1}_{\acco{\norm{Q_t}_F\le\beta_{t+1}}} |\mathcal{F}_{t} \right]\\
	&\le \beta_{t+1}^2 \left(C_{4} + C_{4}\left( F \left( {\theta}_{t,\tau} \right) - F\left(\theta^{*} \right) \right)^2\right)\norm{A_t}_F^2 .
\end{align*}

Then, remark that
\begin{align}
\notag  \frac{1}{\gamma_{t+1}^{2}\beta_{t+1}^{2}}   & \mathbb{E} \left[ 	\norm{\gamma_{t+1}^2A_tW_t\Sigma_{t} A_tW_t\mathbf{1}_{\acco{\norm{Q_t}_F\le\beta_{t+1}}}}^2_F |\mathcal{F}_{t} \right]   \leq \frac{\gamma_{t+1}^{2}}{\beta_{t+1}^{2}} \mathbb{E} \left[ \left\| A_{t}W_{t} \right\|_{F}^{4} |\mathcal{F}_{t} \right] \left\| A_{t} \Sigma_{t} A_{t} \right\|^{2} \\ 
\notag &\quad\le   {\gamma_{t+1}^2 \left(C_{4} + C_{4}\left( F \left( {\theta}_{t,\tau} \right) - F\left(\theta^{*} \right) \right)^2\right)\norm{A_t}_F^2 } \norm{A_t\Sigma_tA_t}^2_F \\
&\notag \quad\le \underbrace{2  {\gamma_{t+1}^2 \left(C_{4} + C_{4}\left( F \left( {\theta}_{t,\tau} \right) - F\left(\theta^{*} \right) \right)^2\right)\norm{A_t}_F^2 }}_{R_{0,t}}\norm{A_t\Sigma_tA_t-I_d}^2_F\\
&\qquad+\underbrace{2  {\gamma_{t+1}^2 \left(C_{4} + C_{4}\left( F \left( {\theta}_{t,\tau} \right) - F\left(\theta^{*} \right) \right)^2\right)\norm{A_t}_F^2 }}_{\tilde{R}_{0,t}}.\label{def::R0}
\end{align}
 {Thanks to Theorem \ref{theo::consistency},  $\norm{A_t}^2_F = \mathcal{O}\pa{t^{1-\gamma}}$, and since $\gamma >3/4$}, one has
\begin{equation*}
\sum_{t\geq 0} R_{0,t}  < + \infty \quad a.s. \quad \quad \text{ and } \quad \quad  \sum_{t\geq 0} \tilde{R}_{0,t}  < + \infty \quad a.s.
\end{equation*}
\end{proof}

\begin{proof}[Proof of Lemma \ref{lem2}]
First, note that since $\left\| A_{t} W_{t} - I_{d} \right\|_{F}^{2} \leq 2 \left\| A_{t} W_{t} \right\|_{F}^{2} +2 \left\| I_{d} \right\|_{F}^{2} \leq 2 \left\| A_{t} \right\|_{F}^{2} \left\|  W_{t} \right\|_{F}^{2} +2 d$, it comes
\begin{align*} 
  \mathbb{E}\cro{\norm{A_tW_t-I_d}^2_F\mathbf{1}_{\acco{\norm{Q_t}_F\le\beta_{t+1}}}| \mathcal{F}_t} &\leq  
 \mathbb{E}\cro{\norm{A_tW_t-I_d}^2_F | \mathcal{F}_t} \\
 &\leq 2\| A_{t}\|_{F}^{4} \mathbb{E}\left[ \| g_{t+1}\pa{\theta_{t,\tau}} g_{t+1}\pa{\theta_{t,\tau}}^{T} \|^2_{F} |\mathcal{F}_{t} \right] + 2d.
\end{align*}
Thanks to Assumption \ref{ass::moment}, one has 
\begin{align*}
\mathbb{E}& \cro{\norm{A_tW_t-I_d}^2_F| \mathcal{F}_{t}}   \leq  2\| A_{t}\|_{F}^{4} \left(C_{4} + C_{4}\left( F \left( {\theta}_{t,\tau} \right) - F\left(\theta^{*} \right) \right)^2\right) + 2d .
\end{align*}
Observe that $\Sigma_{t}$ converges almost surely to $\Sigma$ which is positive, so that
\begin{align*}
\notag \left\| A_{t} \right\|_{F}^{4} &  \leq \left\| A_{t} \right\|_{F}^{4}\mathbf{1}_{\lambda_{\min}(\Sigma_{t}) \geq \lambda_{\min}(\Sigma /2)} + \left\| A_{t} \right\|_{F}^{4}\mathbf{1}_{\lambda_{\min}(\Sigma_{t}) > \lambda_{\min}(\Sigma /2)} \\
& \leq  \frac{1}{\lambda_{\min}\left( \Sigma_{t} \right)^{2}}\left\| A_{t} \right\|_{F}^{4}\mathbf{1}_{\lambda_{\min}(\Sigma_{t}) \geq \lambda_{\min}(\Sigma /2)} + \left\| A_{t} \right\|_{F}^{4}\mathbf{1}_{\lambda_{\min}(\Sigma_{t}) > \lambda_{\min}(\Sigma /2)}
\end{align*}	
Then, since $\| A \|_{F}^{2} \leq 2\| A-B \|_{F}^{2} + 2\| B \|^{2}$,
\begin{align}\label{maj::an::sauvage}
\notag \left\| A_{t} \right\|_{F}^{4} & \leq \frac{4}{\lambda_{\min}(\Sigma)^{2}}\left\| A_{t} \Sigma_{t} A_{t} \right\|_{F}^{2} + \left\| A_{t} \right\|^{4}  \mathbf{1}_{\lambda_{\min}(\Sigma_{t}) <\lambda_{\min}(\Sigma)/2} \\
& \leq \frac{8}{\lambda_{\min}(\Sigma)^{2}}\left\| A_{t} \Sigma_{t} A_{t} - I_{d} \right\|_{F}^{2} + \frac{8}{\lambda_{\min}(\Sigma)^{2}}d  + \left\| A_{t} \right\|^{4}  \mathbf{1}_{\lambda_{\min}(\Sigma_{t}) <\lambda_{\min}(\Sigma)/2},
\end{align}
and
\begin{align}
\label{def::R1::tilde}\mathbb{E} & \cro{\norm{A_tW_t-I_d}^2_F\mathbf{1}_{\acco{\norm{Q_t}_F\le\beta_{t+1}}}| \mathcal{F}_{t}}   \leq \overbrace{ \left(C_{4}+C_{4} \left( F \left(  \theta_{t,\tau} \right)  - F\left(  \theta^{*} \right) \right)^{2} \right)  \frac{16}{\lambda_{\min}^2(\Sigma)}}^{=: \tilde{R}_{1,t}} \| A_{t}\Sigma_{t}A_{t} - I_d \|_{F}^2 \\
\label{def::R2::tilde} &  +  \underbrace{\mathcal{C}_t  \frac{16}{\lambda_{\min}(\Sigma)^2}d +2d  +  2\| A_{t} \|_{F}^{4} \mathbf{1}_{\lambda_{\min}(\Sigma_{t}) <  \lambda_{\min}(\Sigma)/2}\mathcal{C}_t}_{=:\tilde{R}_{2,t}} 
\end{align}
where $\mathcal{C}_t = C_{4} + C_{4}\left( F \left( {\theta}_{t,\tau}\right) - F \left( \theta^{*} \right) \right)^{2}$.
Since $\mathbf{1}_{\lambda_{\min}(\Sigma_{t}) < \lambda_{\min}(\Sigma)/2}$  converges almost surely to $0$
\[
\sum_{t \geq 1} \| A_{t} \|_{F}^{4} \frac{1}{\gamma_{t+1}^{2}\beta_{t+1}^{2}} \mathbf{1}_{\lambda_{\min}(\Sigma_{t}) <  \lambda_{\min}(\Sigma)/2}(C_4 + C_4\left( F \left( {\theta}_{t,\tau} \right) - F\left(\theta^{*} \right) \right)^2) < + \infty \quad a.s.
\]
Then, 
\begin{align}
\label{def::R1}\frac{1}{\gamma_{t+1}^{2}\beta_{t+1}^{2}} \mathbb{E}\left[ \left\| \gamma^2_{t+1}\Sigma_{t}\pa{A_tW_t-I_d}\mathbf{1}_{\acco{\norm{Q_t}_F\le\beta_{t+1}}}\right\|_{F} |\mathcal{F}_{t} \right] &  \leq \overbrace{\frac{\gamma_{t+1}^{2}}{\beta_{t+1}^{2}}\left\| \Sigma_{t} \right\|_{F}^{2} \tilde{R}_{1,t}}^{=: R_{1,t}} \left\| A_{t}\Sigma_{t}A_{t} - I_{d} \right\|_{F} \nonumber \\
& + \underbrace{\frac{\gamma_{t+1}^{2}}{\beta_{t+1}^{2}}\left\| \Sigma_{t} \right\|_{F}^{2} \tilde{R}_{2,t}}_{=: R_{2,t}}
\end{align}
with 
\[
\sum_{t \geq 0} R_{1,t} < + \infty \quad a.s \quad \quad \text{ and } \quad \quad \sum_{t \geq 0} R_{2,t} < + \infty \quad a.s 
\]
\end{proof}

\subsection{Proof of Lemma \ref{lemlink}}\label{sec::lemlink}

Observe that
\begin{align*}
	\norm{A_{t}\Sigma_{t}A_{t} - I_{d}}_F &\le \norm{A_t\Sigma_{t-1}A_t - I_{d}}_F + \norm{A_t\pa{\Sigma_{t-1}-\Sigma_{t}}A_t}_F.
\end{align*}
In addition, thanks to Assumption \ref{ass::sigma::lip}
\begin{align*}
	\norm{A_t\pa{\Sigma_{t-1}-\Sigma_{t}}A_t}_F^{2} & \le \left\| A_{t} \right\|_{F}^{4} \left\| \Sigma_{t-1} - \Sigma_{t} \right\|_{F}^{2}  \\ 
	&  \leq \left\| A_{t} \right\|_{F}^{4} L_{\Sigma} \left\| \theta_{t,\tau} - \theta_{t-1,\tau} \right\|^{2}  \\
	& \leq  \left\| A_{t} \right\|_{F}^{4} L_{\Sigma} \frac{2\ln t^{2\tau }}{\left( \sum_{k=0}^{t-1}\ln (k+1)^{\tau }\right)^{2}} \left( \left\| \theta_{t-1,\tau} - \theta^{*} \right\|^{2} + \left\| \theta_{t} - \theta^{*} \right\|^{2} \right) 
\end{align*}
With the same arguments as for inequality \eqref{maj::an::sauvage}, it comes
\begin{align*}
	\norm{A_t\pa{\Sigma_{t-1}-\Sigma_{t}}A_t}_F^{2}& \leq \left(  \frac{16}{\lambda_{\min}(\Sigma)^{2}}\left\| A_{t-1} \Sigma_{t-1} A_{t} - I_{d} \right\|_{F}^{2} + \frac{16}{\lambda_{\min}(\Sigma)^{2}}d  + \left\| A_{t} \right\|^{4}  \mathbf{1}_{\lambda_{\min}(\Sigma_{t-1}) <\lambda_{\min}(\Sigma)/2} \right) \\
	& \times L_{\Sigma} \frac{2\ln t^{2\tau }}{\left( \sum_{k=0}^{t-1}\ln (k+1)^{\tau }\right)^{2}} \left( \left\| \theta_{t-1,\tau} - \theta^{*} \right\|^{2} + \left\| \theta_{t} - \theta^{*} \right\|^{2} \right) 
	\end{align*}
In order to avoid problems in application of Robbins-Siegmund Theorem, we  now have  to give a first rate of converence of $\theta_{t}$ and $\theta_{t,\tau}$.
This is given by the following lemma.
\begin{lem}\label{lem4}
Under the same assumptions as in Theorem \ref{theo::rate},
for all $\mu < 2 \gamma  + 2\nu - 3$, 
\[
\left\| \theta_{t} - \theta^{*} \right\|^{2} = \mathcal{O}\left( t^{-\mu} \right) \quad a.s. \quad \quad \text{and} \quad \quad \left\| \theta_{t,\tau} - \theta^{*} \right\|^{2} = \mathcal{O}\left( t^{-\mu} \right) \quad a.s.
\]
\end{lem}
The proof is given in Section \ref{sec::prooflem4}.

\paragraph*{Upper bound of $\left\| A_{t}\Sigma_{t}A_{t} - I_{d} \right\|_{F}^{2}$.}
For all $\mu < 2 \gamma  + 2\nu - 3$, since for all $a,b,c$ with $c> 0$, one has $(a+b)^{2} \leq (1+c)a^{2} + (1+c^{-1})b^{2}$, it comes 
\[
\left\| A_{t}\Sigma_{t}A_{t} - I_{d} \right\|_{F}^{2} \leq \left( 1 + \frac{1}{t^{1+\mu/2}} \right) \left\| D_{t} \right\|^{2} + \left( 1+ t^{1+\mu /2} \right) \norm{A_t\pa{\Sigma_{t-1}-\Sigma_{t}}A_t}_F^{2}
\]
it comes
\begin{align*}
\left\| A_{t}\Sigma_{t}A_{t} - I_{d} \right\|_{F}^{2} \leq \left( 1+ R_{3,t} \right) \left\| D_{t} \right\|_{F}^{2} + R_{4,t}
\end{align*}
with
\begin{align}
\label{def::R3} R_{3,t} & = \frac{1}{t^{1+\mu/2}} +  \frac{16L_{\Sigma}}{\lambda_{\min}(\Sigma)^{2}} t^{1+\mu /2 }L_{\Sigma} \frac{\ln t^{2\tau }}{\left( \sum_{k=0}^{t-1}\ln (k+1)^{\tau }\right)^{2}} \left( \left\| \theta_{t-1,\tau} - \theta^{*} \right\|^{2} + \left\| \theta_{t} - \theta^{*} \right\|^{2} \right)   \\
\notag   R_{4,t}& = \left( 1+ t^{1+\mu/2} \right)\left( \frac{16}{\lambda_{\min}(\Sigma)^{2}}d^{2} + \left\| A_{t} \right\|^{4}  \mathbf{1}_{\lambda_{\min}(\Sigma_{t-1}) <\lambda_{\min}(\Sigma)/2} \right)   \\
\label{def::R4} & \times L_{\Sigma} \frac{2\ln t^{2\tau }}{\left( \sum_{k=0}^{t-1}\ln (k+1)^{\tau }\right)^{2}} \left( \left\| \theta_{t-1,\tau} - \theta^{*} \right\|^{2} + \left\| \theta_{t} - \theta^{*} \right\|^{2} \right) 
\end{align}
and it comes from Lemma \ref{lem4} that
\[
\sum_{t \geq 1} R_{3,t} < + \infty \quad a.s \quad \quad \text{ and } \quad \quad \sum_{t \geq 1} R_{4,t} < + \infty \quad a.s 
\]

\subsection{Proof of Lemma \ref{lemk2}}\label{sec::lemk2}
\paragraph*{Upper bound of $K_{2,t}$.}
Thanks to Cauchy–Schwarz inequality, we have 
\begin{align*}
	\mathbb{E} \left[ \left| K_{2,t} \right| |\mathcal{F}_{t} \right] \le 2\gamma_{t+1}\norm{A_t}_F\norm{\Sigma_{t}}_F\norm{A_t\Sigma_tA_t-I_d}_F\mathbb{E}\cro{\norm{A_tW_t-I_d}_F\mathbf{1}_{\acco{\norm{Q_t}_F>\beta_{t+1}}} | \mathcal{F}_t} 
\end{align*}
In addition, since $\left\| A_{t}W_{t} - I_{d} \right\|_{F} \leq \left\| A_{t}W_{t} \right\|_{F} + \left\|  I_{d} \right\|_{F} \leq   \left\| A_{t} \right\|_{F} \left\| W_{t} \right\|_{F} + \sqrt{d} $,
\begin{align*}
	\mathbb{E}\cro{\norm{A_tW_t-I_d}_F\mathbf{1}_{\acco{\norm{Q_t}_F>\beta_{t+1}}} | \mathcal{F}_t} &\le \norm{A_t}^2\mathbb{E}\cro{\norm{g_{t+1}\pa{\theta_{t,\tau}}g_{t+1}^T\pa{\theta_{t,\tau}}}_F\mathbf{1}_{\acco{\norm{Q_t}_F>\beta_{t+1}}} | \mathcal{F}_t} \\
	&\quad + \sqrt{d}\mathbb{P}\cro{\norm{Q_t}_F>\beta_{t+1}| \mathcal{F}_t}
\end{align*}
With the help of Assumption \ref{ass::moment} and Markov's inequality, since $\| Q_{t} \|_{F} \leq \| A_{t} \|_{F} \| g_{t+1} (\theta_{t,\tau} ) \|^{2}$, since $\theta_{t,\tau}$ converges almost surely to $\theta^{*}$, and thanks to Theorem \ref{theo::consistency},
\begin{align*}
	\mathbb{P}\cro{\norm{Q_t}_F>\beta_{t+1}| \mathcal{F}_t} \le \frac{\mathbb{E}\cro{\norm{Q_t}^p_F| \mathcal{F}_t}}{\beta_{t+1}^p}\le \frac{ \norm{A_t}_{F}^{p}\mathbb{E}\cro{\norm{g_{t+1}\pa{\theta_{t,\tau}}}^{2p} | \mathcal{F}_t}}{\beta_{t+1}^p} = \mathcal{O}(t^{p(1-\gamma-\beta)}) \quad a.s..
\end{align*}
In a same way, one can check that
\begin{align*}
\mathbb{E}\cro{\norm{g_{t+1}\pa{\theta_{t,\tau}}g_{t+1}^T\pa{\theta_{t,\tau}}}_F\mathbf{1}_{\acco{\norm{Q_t}_F>\beta_{t+1}}} | \mathcal{F}_t} = \mathcal{O}(t^{p(1-\gamma-\beta)/2}) \quad a.s.
\end{align*}
Then, with the help of equality \eqref{maj::an},
\begin{equation}\label{maj::K2}
R_{5,t} := \mathbb{E} \left[ \left| K_{2,t} \right| | \mathcal{F}_{t} \right] =  \mathcal{O}(n^{3-4\gamma + p(1-\gamma-\beta)/2}) \quad a.s.
\end{equation}
Note that $p>\frac{8-8\gamma}{\gamma+\beta-1}$ gives us
$3-4\gamma + p(1-\gamma-\beta)/2<-1.$

\subsection{Proof of Lemma \ref{lem4}}\label{sec::prooflem4}
\begin{proof}[Proof of Lemma \ref{lem4}.]
With the help of a Taylor's expansion of the functional $F$ and thanks to Assumption \ref{ass::hess}, we obtain, denoting $V_{t} = F ( \theta_{t} ) - F ( \theta^{*})$,
\begin{align*}
  V_{t+1}  &  \leq V_{t} +    \nabla F \left( \theta_{t} \right)^{T} \left( \theta_{t+1} - \theta_{t} \right)  + \frac{L_{\nabla F}}{2}    \left\|   \theta_{t+1} - \theta_{t} \right\|^{2} \\
  &   = V_{t} - \nu_{t+1} \nabla F \left( \theta_{t} \right)^{T}A_{t,\tau'} g_{t+1}\left( \theta_{t} \right)  + \frac{L_{\nabla F}}{2}\nu_{t+1}^{2}   \left\|  A_{t,\tau '}g_{t+1}\left( \theta_{t} \right) \right\|^{2}   .
\end{align*}
Taking the conditional expectation, it comes
\begin{align*}
\mathbb{E} \left[ V_{t+1} |\mathcal{F}_{t} \right] \leq V_{t} - \nu_{t+1} \nabla F \left( \theta_{t} \right)^{T}A_{t,\tau'} \nabla F \left( \theta_{t} \right) + \frac{L_{\nabla F}}{2}\nu_{t+1}^{2} \mathbb{E}\left[ \left\|  A_{t,\tau '}g_{t+1}\left( \theta_{t} \right) \right\|^{2} |\mathcal{F}_{t} \right] .
\end{align*}
Then, thanks to Assumption \ref{ass::moment}
\begin{align*}
	\mathbb{E}\left[V_{t+1} |\mathcal{F}_t\right] &\le \left(1+\frac{L_{\nabla F}C_{p}^{1/p}}{2}\nu_{t+1}^2\lambda_{\max}(A_{t,\tau '})^2\right)V_{t}
	 - \nu_{t+1}\nabla F(\theta_t)^T A_{t,\tau '} \nabla F(\theta_t) \\
	 &  + \frac{L_{\nabla F}C_{p}^{1/p}}{2}\nu_{t+1}^2\lambda_{\max}(A_{t,\tau '})^2.
\end{align*}
Thanks to Assumption \ref{ass1}, the functional $G$
 is locally strongly convex, meaning that there exists a neighborhood of $\theta^{*}$  where the smallest eigenvalues of the Hessian are uniformly bounded below by a positive constant. Then, the Polyak-Lojasiewicz inequality is locally satified, i.e there are positive constants  $r_{0},c_{0}$ such that for all $\theta \in \mathcal{B}\left( \theta^{*} , r_{0} \right)$,  
\[
\left\| \nabla G (\theta ) \right\|^{2} \geq c_{0} \left( F(\theta) - F(\theta^{*}) \right)
\] 
Then, since $\mathbf{1}_{\left\| \theta_{t} - \theta^{*} \right\| \leq r_{0}} = 1-  \mathbf{1}_{\left\| \theta_{t} - \theta^{*} \right\| > r_{0}}$, it comes
\begin{align*}
\|\nabla F(\theta_t)\|^2 & \ge  c_{0} \left(F(\theta_t)-F(\theta^{*})\right)\mathbf{1}_{\left\| \theta_{t} - \theta^{*} \right\| \leq  r_{0}} \\
& \ge  c_{0} \left(F(\theta_t)-F(\theta)\right) - c_{0} \left(F(\theta_t)-F(\theta^{*})\right)\mathbf{1}_{\left\| \theta_{t} - \theta^{*} \right\| > r_{0}}.
\end{align*}
Given \(2\gamma+2\nu-2>1\), there exists \(\mu>0\) such that \(\mu < 2\gamma+2\nu-3\). We define \(\widetilde{V}_t := t^{\mu}V_{t}\), thus
\begin{align*}
	\mathbb{E}\left[\widetilde{V}_{t+1}|\mathcal{F}_t\right] &= \left(\frac{t+1}{t}\right)^\mu\left(\left(1+\frac{L_{\nabla F}C }{2}\nu_{t+1}^2\lambda_{\max}(A_{t,\tau '})^2\right) - c_0\nu_{t+1}\lambda_{\min}(A_{t,\tau '})\right)\widetilde{V}_t\\
	&+ \frac{L_{\nabla F}C}{2}\nu_{t+1}^2\lambda_{\max}(A_{t,\tau '})^2(t+1)^\mu + c_{0}\lambda_{\min}(A_{t,\tau '})\nu_{t+1}(t+1)^{\mu}{V}_{t}\mathbf{1}_{\left\| \theta_{t} - \theta^{*} \right\| > r_{0}}.
\end{align*}
Let \(\zeta_t :=\left(\frac{t+1}{t}\right)^\mu\left(\left(1+\frac{L_{\nabla F}C}{2}\nu_{t+1}^{2}\lambda_{\max}(A_{t,\tau '})^2\right) - c_0\nu_{t+1} \lambda_{\min}(A_{t,\tau '})\right)\), then
\begin{align*}
	\mathbb{E}\left[\widetilde{V}_{t+1}|\mathcal{F}_t\right] &\le \widetilde{V}_t + \frac{L_{\nabla F}C}{2}\nu_{t+1}^2\lambda_{\max}(A_{t,\tau '})^2(t+1)^\mu + \widetilde{V}_t\mathbf{1}_{\zeta_t>1} \\
	&  + c_{0}\lambda_{\min}(A_{t,\tau '})\nu_{t+1}(t+1)^{\mu} {V}_{t}\mathbf{1}_{\left\| \theta_{t} - \theta^{*} \right\| > r_{0}}.
\end{align*}
As \(\nu+\beta <1\) and with the help of equality \eqref{min::lambda}, it comes  that ${1}_{\zeta_t>1}$ converges almost surely to $0$. Moreover, since $\theta_{t}$ is strongly consistant,  $\mathbf{1}_{\left\| \theta_{t} - \theta^{*} \right\| > r_{0}}$ converges almost surely to $0$. Then, 
\[
\sum_{t\geq 1}\widetilde{V}_t\mathbf{1}_{\zeta_t>1} < + \infty \quad a.s. \quad \quad \text{and} \quad \quad \sum_{t \geq 1}\lambda_{\min}(A_{t,\tau '})\nu_{t+1}(t+1)^{\mu} {V}_{t}\mathbf{1}_{\left\| \theta_{t} - \theta^{*} \right\| > r_{0}} < + \infty \quad a.s.
\]
In addition, thanks to equality \eqref{maj::an} and since $ 2 \gamma +2\nu -\mu -2 > 1$, so that
\[
\sum_{t \geq 1} \nu_{t+1}^2\lambda_{\max}(A_{t,\tau '})^2(t+1)^\mu < + \infty  \quad a.s.
\]  
Applying  Robbins-Siegmund Theorem, it follows that $\tilde{V}_{t}$ converges almost surely to a random finite variable, i.e  
\[
F(\theta_{t})-F(\theta^{*}) = \mathcal{O}(t^{-\mu})
\]
for all \(\mu < 2\gamma+2\nu-3\). Due to the local strong convexity of \(G\) (Assumption \ref{ass1}), and since $\theta_{t,\tau}$ is a convex combination of $\theta_{k}$, $k=0, \ldots ,t$, it leads to 

\begin{equation}\label{prem::vit}
\|\theta_t-\theta^{*}\|^2 = \mathcal{O}(t^{-\mu}) \quad a.s \quad \quad \text{ and } \quad \quad \|\theta_{t,\tau '}-\theta^{*}\|^2 = \mathcal{O}(t^{-\mu}) \quad a.s.
\end{equation}
\end{proof}

\subsection{Proof of Lemma \ref{lem::eig0}}\label{sec::lem::eig0}
We denote $\beta'_{t} = \beta_{1} t^{\frac{1-\gamma}{4}}$ for all $t\ge0$. With the same expression of $A_{t+1}$ that we have seen in \eqref{decompositionAt}, we can prove that
\begin{align*}
	\lambda_{\min}(A_{t+1}) \ge \lambda_{\min}(A_t)\pa{1-\gamma_{t+1}\beta'_{t+1}}+\gamma_{t+1}-\gamma_{t+1}\pa{1+\lambda_{\min}(A_t)\norm{W_t}_F}\mathbf{1}_{\acco{\norm{Q_t}_F>\beta'_{t+1}}}.
\end{align*}
By induction, we have for all $t\ge1$ that
\begin{align*}
	\lambda_{\min}(A_{t}) \ge \prod_{j=1}^{t}\pa{1-\gamma_{j}\beta_{j}'}\lambda_{\min}(A_0)+\sum_{k=1}^{t}\prod_{j=k+1}^t(1-\gamma_j\beta'_{j})\gamma_{k} -\mathcal{V}_t,
\end{align*}
where
$$\mathcal{V}_t:=\sum_{k=1}^{t}\prod_{j=k+1}^t(1-\gamma_j\beta'_{j})\gamma_{k}\pa{1+\lambda_{\min}(A_{k-1})\norm{Q_{k-1}}_F}\mathbf{1}_{\acco{\norm{Q_{k-1}}_F>\beta'_{k}}}.$$
In addition, $\mathcal{V}_{t} = \mathcal{V}_{t}' + \mathcal{M}_{t} $ with
\begin{align*}
\mathcal{V}_t '&:=\sum_{k=1}^{t}\prod_{j=k+1}^t(1-\gamma_j\beta'_{j})\gamma_{k} \mathbb{E}\left[ \pa{1+\lambda_{\min}(A_{k-1})\norm{W_{k-1}}_F}\mathbf{1}_{\acco{\norm{W_{k-1}}_F>\beta'_{k}}} |\mathcal{F}_{k-1} \right] , \\
\mathcal{M}_t& :=\sum_{k=1}^{t}\prod_{j=k+1}^t(1-\gamma_j\beta'_{j})\gamma_{k} \mathcal{E}_{k}
\end{align*}
and $\mathcal{E}_{k}  =   \pa{1+\lambda_{\min}(A_{k-1})\norm{Q_{k-1}}_F}\mathbf{1}_{\acco{\norm{Q_{k-1}}_F>\beta'_{k}}}    - \mathbb{E}\left[ \pa{1+\lambda_{\min}(A_{k-1})\norm{Q_{k-1}}_F}\mathbf{1}_{\acco{\norm{Q_{k-1}}_F>\beta'_{k}}} |\mathcal{F}_{k-1} \right]  $ is a sequence of martingale differences. Then, applying Theorem 6.1 in \cite{cenac2020efficient}, one has since $\| A_{t} \|_{F} = O(1)$ a.s. (see equation \eqref{maj::At::O1}), 
\[
\mathcal{M}_{t}^{2} = O \left( \frac{\gamma_{t}}{\beta_{t}'} \right) \quad a.s.
\]
and this term is negligible since $\frac{\gamma}{2} - \frac{1-\gamma}{8} > \frac{1-\gamma}{4}$ (since $\gamma > 3/7$).
In addition, following the same reasoning as for the upper bound of $R_{2,t}$ and since we now know that $\| A_{t} \|_{F} = O (1)$ a.s. (see equation \eqref{maj::At::O1} in the proof of Lemma \ref{lem4}), one has
\[
\mathbb{E}\left[ \pa{1+\lambda_{\min}(A_{t-1})\norm{Q_{t-1}}_F}\mathbf{1}_{\acco{\norm{Q_{t-1}}_F>\beta'_{t}}} |\mathcal{F}_{t-1} \right] = O \left( t^{-p\beta ' /2} \right) a.s.
\]
and applying Lemma 6.1 in \cite{godichon2024online}, it comes that for any $a_{p} < p \beta '/2$,
\[
\mathcal{V}_{t}' = o \left( t^{-a_{p}} \right) \quad a.s
\]
which is negligible as soon as $p>2$.

Finally, 
\begin{align*}
	\sum_{k=1}^{t}\prod_{j=k+1}^t(1-\gamma_j\beta'_{j})\gamma_{k} &\ge \sum_{k=1}^{t}\frac{1}{\beta'_k}\prod_{j=k+1}^t(1-\gamma_j\beta'_{j})\gamma_{k}\beta'_{k} \\
	&=\sum_{k=1}^{t}\frac{1}{\beta'_k}\pa{\prod_{j=k+1}^t(1-\gamma_j\beta'_{j})-\prod_{j=k}^t(1-\gamma_j\beta'_{j})}\\
	&\ge \frac{1}{\beta_{t}'}\pa{1-\prod_{j=1}^t(1-\gamma_j\beta'_{j})}\\
	&\ge \frac{\gamma_1\beta_1}{\beta_{t}'}.
\end{align*}
Since $\prod_{j=0}^{t}\pa{1-\gamma_{t}\beta_{t}'}\lambda_{\min}(A_0) \ge 0$, we have
$$\frac{1}{\lambda_{\min}(A_t)} = \mathcal{O}(\beta'_{t})= \mathcal{O}(t^{\frac{1-\gamma}{4}})\quad a.s.$$

\subsection{Proof of Theorem \ref{theo::rate::stream}}
The proof is analogous to the one of Theorem \ref{theo::rate}. We so just give the main difference here. Observe that in this case, $W_{t} = ng_{t+1}\left( \theta_{t,\tau} \right) g_{t+1} \left( \theta_{t,\tau} \right)^{T}$ and $g_{t+1}\left( \theta_{t,\tau} \right) := \frac{1}{n}   \sum_{i=1}^{t} \nabla_{\theta}f \left( X_{t+1},\theta_{t,\tau} \right)$. We so have  to consider the filtration defined for all $t \geq 1$ by $\mathcal{F}_{t} = \sigma \left( X_{1,1} , \ldots , X_{1,n} , \ldots ,X_{t,1} , \ldots ,X_{t,n} \right)$.

\paragraph*{New values of $\tilde{R}_{1,t}$ and $\tilde{R}_{2,t}$.} Observe that in the streaming case, one has
\begin{align*}
  \mathbb{E}\cro{\norm{A_tW_t-I_d}^2_F\mathbf{1}_{\acco{\norm{Q_t}_F\le\beta_{t+1}}}| \mathcal{F}_t} & 
 &\leq 2n^{2}\| A_{t}\|_{F}^{4} \mathbb{E}\left[ \| g_{t+1}\pa{\theta_{t,\tau}} g_{t+1}\pa{\theta_{t,\tau}}^{T} \|^2_{F} |\mathcal{F}_{t} \right] + 2d.
\end{align*}
and with the help of Assumption \ref{ass::moment}
\begin{align*}
\mathbb{E}\left[ \left\| g_{t+1}\left( \theta_{t,\tau} \right) g_{t+1} \left( \theta_{t,\tau} \right) \right\|_{F}^{2} |\mathcal{F}_{t} \right] & \leq  \mathbb{E} \left[ \left\| g_{t+1} \left( \theta_{t,\tau} \right) \right\|^{4} |\mathcal{F}_{t} \right] \\
& \leq \left( \frac{1}{n} \sum_{i=1}^{n} \left( \mathbb{E} \left[ \left\|  \nabla_{\theta}\left(f(X_{t+1,i},\theta_{t,\tau}\right) \right\|^{4} |\mathcal{F}_{t}	 \right] \right)^{\frac{1}{4}} \right)^{4} \\ &\leq C_{4} + C_{4}\left( F \left( {\theta}_{t,\tau} \right) - F\left(\theta^{*} \right) \right)^{2}
\end{align*}
Then, in the streaming case, $\tilde{R}_{1,t}$ and $\tilde{R}_{2,t}$ defined by \eqref{def::R1::tilde} and \eqref{def::R2::tilde} have to be replaced by
\begin{align}
  &  \tilde{R}_{1,t} :=   n \left(C_{4}+C_{4} \| \theta_{t,\tau} - \theta^{*} \|^{4} \right)  \frac{16}{\lambda_{\min}^2(\Sigma)} \\
& \tilde{R}_{2,t} :=  n   \left(C_{4}+C_{4} \| \theta_{t,\tau} - \theta^{*} \|^{4} \right)  \frac{16}{\lambda_{\min}(\Sigma)^2}d +2d  +  2\| A_{t} \|_{F}^{4} \mathbf{1}_{\lambda_{\min}(\Sigma_{t}) <  \lambda_{\min}(\Sigma)/2}\left(C_{4} + C_{4}\| {\theta}_{t,\tau} - \theta^{*} \|^4\right) .
\end{align}

\paragraph*{New values in the upper bound of $K_{2,t}$. }
Analogously, the only difference there is that  
\[
\mathbb{P}\left[ \left\| Q_{t} \right\|_{F} > \beta_{t+1} |\mathcal{F}_{t} \right] \leq \frac{n^{p}  \left\| A_{t} \right\|^{p} \left( C_{p} + C_{p} \left(  F \left( \theta_{t,\tau} \right) - F \left( \theta^{*} \right) \right)^{p} \right)}{\beta_{t+1}^{p}} .
\]

\paragraph*{Main difference with the proof of Theorem \ref{theo::rate}.} The main difference results in $\mathbb{E} \left[ K_{1,t} | \mathcal{F}_{t} \right]$. Indeed, in the streaming case, since $X_{t+1,i}$ and $X_{t+1,j}$ are independent as soon as $i \neq j$, it comes
\begin{align*}
\mathbb{E}\left[ n g_{t+1}\left( \theta_{t,\tau} \right) g_{t+1}\left( \theta_{t,\tau} \right)^{T} |\mathcal{F}_{t} \right] & = \frac{1}{n}\sum_{i=1}^{n}\mathbb{E}\left[ \nabla_{\theta}f\left( X_{t+1,i}, \theta_{t,\tau} \right)\nabla_{\theta}f\left( X_{t+1,i}, \theta_{t,\tau} \right)^{T} |\mathcal{F}_{t} \right] \\
&  + \frac{1}{n}\sum_{i=1}^{n}\sum_{j \neq i} \mathbb{E}\left[ \nabla_{\theta}f\left( X_{t+1,i}, \theta_{t,\tau} \right)\nabla_{\theta}f\left( X_{t+1,j}, \theta_{t,\tau} \right)^{T} |\mathcal{F}_{t} \right] \\
&  = \Sigma_{t} + (n-1) \nabla F \left( \theta_{t,\tau }\right) \nabla F \left( \theta_{t,\tau} \right)^{T}.
\end{align*}
Then, in the streaming case, one has
\[
\mathbb{E}\left[ K_{1,t} |\mathcal{F}_{t} \right] \geq \left\langle A_{t} \Sigma_{t} \tilde{D}_{t}, \tilde{D}_{t} \right\rangle_{F} + (n-1) \left\langle A_{t} \Sigma_{t}   A_{t} \nabla F\left( \theta_{t,\tau} \right)\nabla F\left( \theta_{t,\tau} \right)^{T}A_{t}   , \tilde{D}_{t}  \right\rangle_{F}  
\]
Following the same reasoning as in the proof of Theorem \ref{theo::rate}, for all $\mu < 2\gamma +2\mu -3$, 
\[
\left\| \theta_{t} - \theta^{*} \right\|^{2} = o \left( \frac{1}{t^{\mu}} \right) \quad a.s \quad \quad\text{ and } \quad \quad \left\| \theta_{t,\tau} - \theta^{*} \right\|^{2} = o \left( \frac{1}{t^{\mu}} \right) \quad a.s
\] 
and since $\nabla F$ is $L_{\nabla F}$ Lispchitz, 
\[
\left\| \nabla F \left(  \theta_{t} \right) \right\|^{2} = o \left( \frac{1}{t^{\mu}} \right) \quad a.s \quad \quad\text{ and } \quad \quad \left\|  \nabla F \left(  \theta_{t,\tau}  \right) \right\|^{2} = o \left( \frac{1}{t^{\mu}} \right) \quad a.s.
\]
In addition, for all $  \mu '  > 0$, one has
\begin{align*}
\gamma_{t+1}\left| \left\langle A_{t} \Sigma_{t}   A_{t} \nabla F\left( \theta_{t,\tau} \right)\nabla F\left( \theta_{t,\tau} \right)^{T}A_{t}   , \tilde{D}_{t}  \right\rangle_{F} \right|   &\leq      \left( 1+ \frac{1}{t^{1+\mu '}} \right) \left\| \tilde{D}_{t} \right\|_{F}^{2} 
\\ &\quad+ \underbrace{ t^{1+\mu '}  \gamma_{t+1}^{2}  \left\| A_{t} \right\|_{F}^{6} \left\| \Sigma_{t} \right\|_{F} \left\|  \nabla F \left( \theta_{t,\tau} \right) \right\|^{4} }_{=: R_{n,t}}
\end{align*}

Then,
\[
R_{n,t} = o \left( \frac{1}{t^{8\gamma - 7 + 2\mu - \mu ' } }\right) \quad a.s
\]
Taking $\mu > 4-4\gamma$ (since $\mu < 2 \gamma +2\nu -3$, this is possible as soon as $6\gamma +2\nu > 7$) and $\mu ' < 8\gamma - 7 + 2\mu- 1$, one has
\[
\sum_{t \geq 1} R_{n,t}  < + \infty \quad a.s.
\]

\paragraph*{Conclusion} One can so rewrite  the upper bound of $\mathbb{E}\left[ \left\| D_{t+1} \right\|_{F}^{2} |\mathcal{F}_{t} \right] $ given by \eqref{updt} can be replaced by 
\[
\mathbb{E}\left[ \left\| D_{t+1} \right\|_{F}^{2} |\mathcal{F}_{t} \right] \leq  \left( 1+ S_{1,t} \right) \left\| D_{t} \right\|_{F}^{2} + S_{2,t} - \gamma_{t+1} \left\langle A_{t} \Sigma_{t} \tilde{D}_{t}, \tilde{D}_{t} \right\rangle_{F}
\]
with
\begin{align*}
S_{1,t} &  =    \left( \frac{1}{t^{1+\mu '}}+  16\gamma_{t+1}^{2}\beta_{t+1}^{2} \left( R_{0,t} + 2R_{1,t} \right) + 8\left( R_{0,t}+ 2R_{1,t} \right) \right)\left( 1+ R_{3,t} \right) + R_{3,t}  \\
S_{2,t}&  =  R_{n,t}+  16\gamma_{t+1}^{2}\beta_{t+1}^{2} \left( \tilde{R}_{0,t} + 2R_{2,t} \right) +  8\left( \tilde{R}_{0,t}+ 2R_{2,t} \right) + R_{4,t} + R_{5,t}
\end{align*}
and conclude as in the proof of Theorem \ref{theo::rate}.

\bibliographystyle{apalike}
\bibliography{Biblio_Rapport}

\end{document}